\documentclass[twoside,11pt]{article}
\usepackage[english]{babel}
\usepackage{amsmath}
\usepackage{amsthm}
\usepackage{amsfonts}
\usepackage{amssymb}
\usepackage{graphicx}
\usepackage{mathrsfs}

\usepackage{graphicx}

\makeatletter
\newcommand*\bigcdot{\mathpalette\bigcdot@{.5}}
\newcommand*\bigcdot@[2]{\mathbin{\vcenter{\hbox{\scalebox{#2}{$\m@th#1\bullet$}}}}}
\makeatother

 \theoremstyle{plain}
\newtheorem{theorem}{Theorem}[section]
\newtheorem{proposition}[theorem]{Proposition}
\theoremstyle{remark}
\newtheorem{definition}[theorem]{\bf Definition}
\newtheorem{remark}[theorem]{\bf Remark}
\newtheorem{example}[theorem]{\bf Example}
\numberwithin{equation}{section}

\newcommand{\be}{\begin{equation}}
\newcommand{\ee}{\end{equation}}

\allowdisplaybreaks[3]

\begin{document}

\title{{\LARGE Dispersion of compressible rotating Euler equations with low Mach and Rossby numbers}
 \footnotetext{\small E-mail addresses: mupc024@163.com
 }
 }

\author{{Pengcheng Mu}\\[2mm]
\small\it School of Mathematical Sciences, Anhui University, Hefei 230601, P.R.China
}

\date{}

\maketitle

\begin{quote}
\small \textbf{Abstract}: In this paper, we consider the low Mach and Rossby number singular limits and longtime existence of strong solution to the initial value problem of 3D compressible rotating Euler equations with ill-prepared initial data. We establish the Strichartz decay estimates that are uniform to the Mach number, the Rossby number, and the ratio of these two parameters for the associated linear propagator without any restrictions on the frequency. In particular, difficulties arisen from the degeneracy of the phase function and the vanishing of the ratio of the two parameters are addressed by elaborately designed splitting techniques and discussions for each frequencies. Using the decay estimates, we prove the longtime existence and obtain a rate of convergence to zero of strong solution to the compressible rotating Euler equations with initial data of finite energy in $\mathbb{R}^3$.


\indent \textbf{AMS (2010) Subject Classification}: 35Q35; 76U05; 76M45
\end{quote}

\section{Introduction}
We study the singular limits and longtime existence of strong solution to the rotating isentropic compressible Euler equations:
\begin{equation}\label{el}
\begin{cases}
\partial_t\rho+\mathrm{div}(\rho u)=0,\\
\partial_{t}(\rho u)+\mathrm{div}(\rho u\otimes u)+\frac{1}{\varepsilon}\mathbf{e}_3\times \rho u+\frac{1}{\delta^2}\nabla P(\rho)=0.
\end{cases}
\end{equation}
In \eqref{el}, the unknowns $\rho$ and $u=(u_1,u_2,u_3)^{\top}$ are the density and velocity of the fluid, respectively. The pressure $P(\rho)$ is given by $P(\rho)=\gamma^{-1}\rho^{\gamma}$ with $\gamma>1$, and $\mathbf{e}_3=(0,0,1)^{\top}$. The independent variables are $t\in\mathbb{R}^{+}$ and $x=(x_1,x_2,x_3)\in\mathbb{R}^3$. The system \eqref{el} is written in non-dimensional form where $\varepsilon$ is the Rossby number representing the ratio of the displacement due to inertia to the displacement duo to Corilis force, and $\delta$ is the Mach number denoting the ratio of the characteristic speed of the fluid to the speed of sound. In this paper, we consider the low Mach and Rossby number singular limits of equations \eqref{el} with ill-prepared initial data in the following regimes:
\begin{equation}\label{lim}
\begin{split}
&(1).~\varepsilon\to0,~~\delta\to0,~~\mathrm{while}~~\frac{\delta}{\varepsilon}=\mathrm{constant};\\
&(2).~\varepsilon\to0,~~\delta\to0,~~\mathrm{and}~~\lim_{\varepsilon,\delta\to0}\frac{\delta}{\varepsilon}=0,
\end{split}
\end{equation}
and prove the longtime existence of the solution. The first regime in \eqref{lim} is a classical two-scale singular limit, while the second limit is a three-scale one because the terms containing small parameters $\varepsilon$ and $\delta$ in \eqref{el} induce fast waves moving on each of the time scales $O(\varepsilon^{-1})$ and $O(\delta^{-1})$, and a slow mode moving on the $O(1)$ time scale is also present. The studies of limits in \eqref{lim} are relevant because in large scale motion in geophysics the Mach number and the Rossby number are both very small. Now, by using the substitution
\begin{equation*}
\rho=\big(\bar{\gamma}(1+\delta b)\big)^{1/\bar{\gamma}},~~\bar{\gamma}=\frac{\gamma-1}{2},
\end{equation*}
we rewrite \eqref{el} as
\begin{equation}\label{CEL}
\partial_t U+\frac{\bar{\gamma}}{\delta}\mathcal{L}U+\mathcal{N}(U,U)=0,
\end{equation}
where
\begin{equation}\label{e9182}
U=\left[\begin{matrix}
b\\u
\end{matrix}\right],~~
\mathcal{L}U=\left[\begin{matrix}
\mathrm{div}\,u\\
\frac{\delta}{\bar{\gamma}\varepsilon}\mathbf{e}_3\times u+\nabla b
\end{matrix}\right],~~\mathrm{and}~~
\mathcal{N}(U,U)=\left[\begin{matrix}
u\cdot\nabla b+\bar{\gamma}b\cdot\mathrm{div}\,u\\
u\cdot\nabla u+\bar{\gamma}b\cdot\nabla b
\end{matrix}\right].
\end{equation}
We supplement \eqref{CEL} with initial data
\begin{equation}\label{in}
U|_{t=0}=U_0=(b_0,u_0)^{\top}.
\end{equation}
For convenience, we set
\begin{equation}\label{8202}
\nu:=\frac{\delta}{\bar{\gamma}\varepsilon},
\end{equation}
then in the first limit regime in \eqref{lim} $\nu$ remains positive, while in the second limit $\nu$ reduces to zero as $\varepsilon,\delta\to0$.

It is well-known that, in whole space the propagation of the fast oscillating waves---the main difficulty involved in the singular limit problems with ill-prepared data---can be dealt with provided
that suitable dispersive (Strichartz) estimates can be obtained. We now recall some important previous works related to the dispersion and longtime existence of solutions to rotating fluids. In the case of incompressible rotating fluids, we first mention the works of J-Y. Chemin, B. Desjardins, I. Gallagher and E. Grenier \cite{C-book,B8151,J-Y}. In these works the authors derived the dispersion estimates for wave equations related to the rotating Navier-Stokes equations, and studied the fast rotation limit of the equations with initial data that is the sum of a 2D part $v_0$ and a 3D part $w_0$ belonging to the anisotropic Sobolev spaces. It was shown in \cite{C-book,B8151,J-Y} that the rotating Navier-Stokes system admits a unique global solution as the Rossby number is small enough, and the solution converges to that of the incompressible 2D Navier-Stokes equations with the initial data $v_0$ as the Rossby number goes to zero. Since then, there arise plentiful works concerning the dispersion and global existence of solutions to the incompressible viscous rotating fluids such as \cite{B8158, B8156, B8157, B8153, B8152}. For incompressible inviscid fluids, since there is no smoothing effect derived from the viscosity, the global existence of strong solutions can not be obtained even in the fast rotation setting. However, A. Dutrifoy \cite{B8159} proved  that the solution of the incompressible rotating inviscid fluid can exist in arbitrary time intervals as the Rossby number $\varepsilon$ is small enough, and gave a lower bound on the lifespan $T_{\varepsilon}$ of the solution as $T_{\varepsilon}\gtrsim \ln\ln\varepsilon^{-1}$. In \cite{b-Koh-2014-JDE}, Y. Koh, S. Lee and R. Takada obtained the optimal range of the Strichartz estimate for the linear propagator associated with the rotating Euler equation, and improved the maximal lifetime of the strong solution to $O(\ln \varepsilon^{-1})$. We also refer the readers to \cite{Guo-2023-CPAM,Jia-2020,G.S-2022,b-Takada-2016-Jap} for studies in this direction.

As for the compressible rotating fluids, E. Fereisl, I. Gallagher and A. Novotn\'{y} \cite{B-F6} first investigated the asymptotic in the first limit in \eqref{lim} of global weak solutions to the compressible rotating Navier-Stokes equations with ill-prepared data and complete slip boundary conditions in an infinite slab $\mathbb{R}^2\times (0,1)$. Using the celebrated RAGE theorem, they proved the dispersion and justify the convergence towards the 2D viscous quasi-geostrophic equation of the original system. After that, the incompressible and fast rotation limits of compressible rotating fluids in infinite slab were studied in many works, such as \cite{B-F4,B-F2,B-F1,B-F3,B-K1} for viscous fluids, and \cite{B-C1,B-T1} for inviscid fluids. For the studies in whole space, we first recall the work of M. Caggio and \v{S}. Ne\v{c}asov\'{a} \cite{M.Ca}. By combining the relative entropy method, the dispersion of acoustic wave, and the fact that fast rotation can weaken the nonlinear effects in Euler equation and thus extend the lifespan of the solution, the convergence of weak solutions of rotating isentropic compressible Navier-Stokes system to the strong solution of rotating incompressible Euler equation on arbitrary time interval was proved in \cite{M.Ca} as the Mach number and Reynolds number tend to zero and the rotation is strong enough. However, the dispersion as the Rossby number tends to zero was not explored in \cite{M.Ca}. As far as we know, the dispersive phenomenon of the compressible rotating fluids in 3D whole space as the Rossby number and Mach number tend to zero simultaneously was first considered by V-S. Ngo and S. Scrobogna in \cite{Ngo-DCDS}. In that paper, the authors constructed the Strichartz-type estimates for the linear system of \eqref{CEL} in which the Rossby number and the Mach number are identical and the frequency is localized in $\mathcal{C}_{r,R}=\{\xi\in\mathbb{R}^3:|\xi|\leq R,|\xi_h|\geq r,|\xi_3|\geq r\}$. They also proved the longtime existence of strong solution to \eqref{CEL} with initial data of finite energy in $\mathbb{R}^3$ by a detailed study on a 3D nonlinear hyperbolic system as $r\ll 1$ and $R \gg1$. Indeed, by a study on the spectra of $\mathcal{L}$ and Duhamel's principle, the solution to \eqref{CEL} can be written through the propagator
\begin{equation}\label{8151}
e^{i\frac{t}{\delta}p(D)}f(x):=\frac{1}{(2\pi)^{3}}\int_{\mathbb{R}^{3}}{e^{ix\cdot\xi+i\frac{t}{\delta}p(\xi)}\hat{f}(\xi)}d\xi,~~(t,x)\in\mathbb{R}^{1+3},
\end{equation}
where $\hat{f}$ is the Fourier transform of $f$, and the phase function $p(\xi)$ is given by:
\begin{equation}\label{phas}
p(\xi)=a\cdot(\tilde{A}\pm\tilde{B})
\end{equation}
with
\begin{equation*}
\begin{split}
&\tilde{A}=\sqrt{|\xi|^2+\nu^2+2\nu\xi_3},~~\tilde{B}=\sqrt{|\xi|^2+\nu^2-2\nu\xi_3},\\
&\xi=(\xi_1,\xi_2,\xi_3),~~|\xi|=\sqrt{\xi^2_1+\xi_2^2+\xi_3^2},~~\mathrm{and}~~a=\mathrm{constant}.\\
\end{split}
\end{equation*}
Denoting by $D^2p$ the Hessian matrix of $p$. Then by a detailed calculation we can check that
\begin{equation}\label{e9291}
\mathrm{det}\,D^2 p(\xi)=\frac{16\nu^3\cdot\xi_3\cdot|\xi_h|^2\cdot a^3}{\tilde{A}^4\tilde{B}^4(\tilde{A}\mp \tilde{B})},~~\mathrm{for}~~p(\xi)=a\cdot(\tilde{A}\pm\tilde{B}).
\end{equation}
Thus, the phase $p(\xi)$ is actually degenerated in $\{\xi_h=0\}\cup\{\xi_3=0\}$. In \cite{Ngo-DCDS} (Theorem 4.1), for $\varepsilon=\delta$, $q\in[2,\infty]$, $p\geq 4q/(q-2)$ and $f\in L^{2}(\mathbb{R}^3)$ such that $\mathrm{supp}\,\hat{f}\subset \mathcal{C}_{r,R}$, the authors obtained that
\begin{equation*}
\|e^{i\frac{t}{\delta}p(D)}f\|_{L^p(\mathbb{R}^+;L^q(\mathbb{R}^3))}\leq C R^{\frac{3}{2}-\frac{3}{q}}r^{-\frac{4}{p}}\varepsilon^{\frac{1}{p}}\|f\|_{L^{2}(\mathbb{R}^3)}.
\end{equation*}
Moreover, the maximum lifetime of solution to \eqref{CEL} was shown to be $O(\varepsilon^{-\alpha})$ with $\alpha$ being a positive constant in \cite{Ngo-DCDS}.

In this paper, our main purposes are to study the dispersion and asymptotic of strong solution to \eqref{CEL} for both of the two limits in \eqref{lim} with ill-prepared data of finite energy in $\mathbb{R}^3$, and prove the longtime existence of the solution. The key point here is to construct the dispersive decay estimates for the propagator \eqref{8151}. Different from \cite{Ngo-DCDS}, we shall establish the dispersive estimates for \eqref{8151} {\it for all frequencies}. Thus difficulties arisen from the degeneracy of the phase have to be addressed. Furthermore, to the best of our knowledge, this is the first work concerning the dispersion of three-scale singular limit of 3D compressible rotating fluids. Indeed, in the second limit of \eqref{lim}, the phase function $p(\xi)$ involves a small parameter $\nu$ and the corresponding determinant in \eqref{e9291} shall vanish as $\nu$ tends to zero. Thus, besides covering the frequency domain with degeneracy, the dispersive decay estimates constructed in this paper have also to be {\it uniform in $\nu$}. Now, by using the classical theory established by Klainerman and Majda \cite{K-M-1,K-M-2} (or \cite{M-book}), we have:
\begin{proposition}\label{pro-ex}
Assume that $U_0\in H^{m}(\mathbb{R}^3)$ with $m\geq3$. Then for any $\varepsilon>0$ and $\delta>0$, the initial value problem \eqref{CEL}-\eqref{in} admits a unique strong solution $U=(b,u)^{\top}\in C([0,T^*];H^{m}(\mathbb{R}^{3}))$ for some $T^*>0$, and the following estimate holds:
\begin{equation}\label{est}
\|U\|_{L^{\infty}(0,T^*;H^m(\mathbb{R}^3))}\leq M_1\|U_0\|_{H^m(\mathbb{R}^3)},
\end{equation}
where the constant $M_1$ and the time $T^*$ are independent of $\varepsilon$ and $\delta$.
\end{proposition}

We now state the main results of this paper.
\begin{theorem}\label{th-main}
Assume that $m\geq4$ and $U\in C([0,T^*];H^{m}(\mathbb{R}^{3}))$ be the strong solution of \eqref{CEL}-\eqref{in} given by Proposition \ref{pro-ex}. Then for any $l\in\{0,1\}$ and $q\in[4,\infty)$, it holds:
\begin{equation}\label{8124}
\begin{split}
&\|\nabla^l U\|_{L^q(0,T^*;L^{\infty}(\mathbb{R}^3))}\leq M_2\varepsilon^{\frac{1}{q}}\big(\|U_0\|_{H^{m}(\mathbb{R}^3)}+\|U_0\|^2_{H^{m}(\mathbb{R}^3)}\big),
\end{split}
\end{equation}
where the constant $M_2$ may depend on $T^*$ but is independent of $\varepsilon$ and $\delta$. Furthermore, there exists $\varepsilon_0>0$ such that as $\varepsilon\in(0,\varepsilon_0]$,
\begin{equation}\label{8125}
T^*\geq\bigg(\frac{M_3}{\|U_0\|_{H^{m}(\mathbb{R}^3)}+\|U_0\|^2_{H^{m}(\mathbb{R}^3)}}\bigg)^{4/7}\frac{1}{\varepsilon^{1/7}},
\end{equation}
where the constant $M_3$ is independent of $\varepsilon$ and $\delta$.
\end{theorem}

The paper is arranged as follows. First, the proofs of Theorem \ref{th-main} are presented in Section \ref{sec2}. Then, in Section \ref{sec927} we prove Proposition \ref{pro1} which concerns the dispersive decay estimates of the propagator and is key in the proof of Theorem \ref{th-main}. Finally, in Section \ref{sec3} we prove a result used in this paper through the TT* argument, and present some applications of the result. Now we give some remarks on the results and proofs of Theorem \ref{th-main}.

\begin{remark}
The constructions of the dispersive decay estimates in Section \ref{sec927} are complicated. Indeed, we first use a splitting technique on the frequency shell to overcome the difficulties arisen from the degeneracy along $\{\xi_h=0\}\cup\{\xi_3=0\}$. Then, to obtain the estimates that are uniform to $\nu$,
different strategies are designed for high, middle and low frequencies. Furthermore, noting that in this work the phase $p(\xi)$ can be each of $\tilde{A}-\tilde{B}$ and $\tilde{A}+\tilde{B}$. These two functions have different properties in some frequency domains, thus the corresponding discussions also vary.
\end{remark}

\begin{remark}
The initial data $U_0$ in the present work is assumed to have finite energy in $\mathbb{R}^3$. Since the only element that has finite energy in $\mathbb{R}^3$ and belongs to the kernel of $\mathcal{L}$ is zero, the limits of the corresponding solution are also zero as $\varepsilon,\delta\to0$ in both of the two limits in \eqref{lim}. We mention here that, in the more general case where the initial data is the sum of a 2D part that belongs to the kernel of $\mathcal{L}$ and a 3D part with finite energy in $\mathbb{R}^3$, the limits are no more zero and differ for the two limits in \eqref{lim}. However, the expected convergences in this case are not easy to be rigorously justified at least for the studies in 3D whole space.
\end{remark}

\begin{remark}
We compare the conditions and results of Theorem \ref{th-main} with those in \cite{Ngo-DCDS}. Indeed, for $\varepsilon=\delta$ small enough, $s$, $s_0$, $p$, $\beta$, $\kappa$ such that
\begin{equation*}
s>\frac{5}{2},~~s_0>0,~~1<p<2,~~\beta>0,~~\kappa>0,~~\beta(5+2\kappa)<\frac{1}{2},
\end{equation*}
and the initial data $U_{0}^{\varepsilon}$ satisfying
\begin{equation*}
\|U_0^{\varepsilon}\|_{s,s_0,p}:=\max\Big\{\|U_0^{\varepsilon}\|_{H^{s+s_0}(\mathbb{R}^3)},\|U_0^{\varepsilon}\|_{L^{2} (\mathbb{R}^2_{x_{h}};L^{p}(\mathbb{R}_{x_{3}}))},\|U_0^{\varepsilon}\|_{L^{p} (\mathbb{R}^2_{x_{h}};L^{2}(\mathbb{R}_{x_{3}}))}\Big\}<\infty,
\end{equation*}
the authors in \cite{Ngo-DCDS} showed that the system \eqref{CEL} admits a unique strong solution $U^{\varepsilon}$ such that
\begin{equation}\label{10161}
U^{\varepsilon}\in C([0,T^{\star}_{\varepsilon}];H^{s}(\mathbb{R}^3))~~\mathrm{with} ~~T_{\varepsilon}^{\star}\geq\frac{C}{\mathcal{C}(U_0^{\varepsilon})\varepsilon^{\alpha}},
\end{equation}
where
\begin{equation}\label{10162}
\mathcal{C}(U_0^{\varepsilon})=\max\Big\{\|U_0^{\varepsilon}\|_{s,s_0,p},\|U_0^{\varepsilon}\|_{s,s_0,p}^2\Big\},~~\mathrm{and}~~\alpha=\min\Big\{\frac{1}{4}-\frac{\beta(5+2\kappa)}{2},\frac{\beta s_0}{2}\Big\}.
\end{equation}
Compare to $\alpha$ in \eqref{10161}, the exponent $1/7$ in \eqref{8125} of this paper is not affected by the index $s_0$. Indeed, in \cite{Ngo-DCDS} $s_0$ characterises the difference between the differentiability of the initial data and the solution, which in this paper equals to $0$. However, taking $s_0=0$ in \eqref{10162} leads to $\alpha=0$ and as a result the lower bound of $T^{\star}_{\varepsilon}$ is of order $O(1)$ as $\varepsilon$ tends to zero. Furthermore, in this paper, we do not require the initial data to belong to the anisotropy spaces $L^{2} (\mathbb{R}^2_{x_{h}};L^{p}(\mathbb{R}_{x_{3}}))$ and $L^{p} (\mathbb{R}^2_{x_{h}};L^{2}(\mathbb{R}_{x_{3}}))$.

\end{remark}

\begin{remark}
From \eqref{8125} we can see that, as $\|U_0\|_{H^{m}(\mathbb{R}^3)}\leq \varepsilon^{-\omega}$ with $0<\omega<\frac{1}{8}$, then
\begin{equation*}
T^*\geq C\bigg(\frac{1}{\varepsilon^{1-8\omega}}\bigg)^{\frac{1}{7}}\to\infty,~~\mathrm{as}~~\varepsilon\to0.
\end{equation*}
Thus, the initial data $U_0$ can be chosen even to blow up as $\varepsilon$ goes to zero.
\end{remark}

We end this section by appointing some notations that shall be used. First, for a vector $v=(v_1,v_2,v_3)$, we use $v_h=(v_1,v_2)$ to denote its first two components. For a function $f$, we use $\hat{f}$ to denote its Fourier transform, and use $\mathcal{F}^{-1}f$ to denote its inverse. We let $g(D)f(x):=\mathcal{F}^{-1}(g(\xi)\hat{f}(\xi))$ for $g=g(\xi)$, and use $C$ to denote the `irrelevant' constants which may take different values in each appearance.

\section{Proof of Theorem \ref{th-main}}\label{sec2}

We begin from studying the spectra of the large operator $\mathcal{L}$. Taking the Fourier transform of $\mathcal{L}$ yields
\begin{equation*}
\widehat{\mathcal{L}}(\xi)=\left(\begin{matrix}
0&i\xi_{1}&i\xi_{2}&i\xi_3\\
i\xi_{1}&0&-\nu&0\\
i\xi_{2}&\nu&0&0\\
i\xi_3&0&0&0\\
\end{matrix}\right).
\end{equation*}
Solving the eigenvalue problems
\begin{equation}\label{7251}
\widehat{\mathcal{L}}(\xi)\mathbf{r}_{j}(\xi)=i p_{j}(\xi)\mathbf{r}_{j}(\xi),~~j=1,2,3,4,
\end{equation}
we obtain $p_1(\xi)$, $p_2(\xi)$, $p_3(\xi)$ and $p_4(\xi)$ as the following four functions:
\begin{equation}\label{eq10101}
\begin{split}
\pm\frac{1}{2}\Big(\sqrt{|\xi|^2+\nu^2+2\nu\xi_3}\pm\sqrt{|\xi|^2+\nu^2-2\nu\xi_3}\Big),\\
\end{split}
\end{equation}
where $\xi=(\xi_{1},\xi_{2},\xi_3)$ and $|\xi|=\sqrt{\xi_{1}^{2}+\xi_{2}^{2}+\xi_3^2}$. The expressions of the eigenvectors $\mathbf{r}_j(\xi)$ are not important for our use. However, we can see that for $\mathrm{a.e}~\forall\,\xi\in\mathbb{R}^3$, $\{\mathbf{r}_{j}(\xi)\}_{j=1}^{4}$ is an orthonormal basis in $\mathbb{C}^{4}$. Now we define the operators
\begin{equation}\label{274}
\mathcal{P}_{(j)}f(x)=\mathcal{F}^{-1}\big((\widehat{f}\cdot \overline{\mathbf{r}_{j}})\mathbf{r}_{j}\big)(x),~~j=1,2,3,4,
\end{equation}
then for $\forall f\in (L^2(\mathbb{R}^3))^4$ we know that $f=\sum\nolimits_{j=1}^4\mathcal{P}_{(j)}f$. Applying $\mathcal{P}_{(j)}$ to \eqref{CEL}, and using Duhamel's principle yields
\begin{equation}\label{2310}
\mathcal{P}_{(j)}U(t)=e^{-i\frac{\bar{\gamma}t}{\delta}p_{j}(D)}\mathcal{P}_{(j)}U_{0}-\int_{0}^{t}{e^{-i\bar{\gamma}\frac{t-s}{\delta}p_{j}(D)}\mathcal{P}_{(j)}\mathcal{N}(U,U)(s)}ds,~~j=1,2,3,4.
\end{equation}


\subsection{Littlewood-Paley decomposition}
Let $\varphi\in \mathcal{S}(\mathbb{R}^{3})$ be a radial function, supported in the annular $\mathcal{C}:=\{\xi\in\mathbb{R}^{3}:\frac{3}{4}\leq |\xi|\leq \frac{8}{3}\}$ such that
\begin{equation*}
\sum\limits_{k\in\mathbb{Z}}\varphi_{k}(\xi)=1,~~\forall\,\xi\neq0,
\end{equation*}
where $\varphi_{k}(\xi)=\varphi(2^{-k}\xi)$. We define the operators
\begin{equation}\label{235}
\begin{split}
&\triangle_{k}f(x):=\varphi_{k}(D)f(x),~~k\in\mathbb{Z}.
\end{split}
\end{equation}
For $(r,\sigma)\in[1,+\infty]^{2}$, $m\in\mathbb{R}$ and $q\in[1,\infty]$, we define the norm of the standard homogeneous Besov spaces $\dot{B}^{m}_{r,\sigma}(\mathbb{R}^{3})$:
\begin{equation*}
\|f\|_{\dot{B}^{m}_{r,\sigma}}:=\bigg(\sum\limits_{k\in\mathbb{Z}}\big(2^{mk}\|\triangle_{k}f\|_{L^{r}}\big)^{\sigma}\bigg)^{\frac{1}{\sigma}},
\end{equation*}
and the space-time norm:
\begin{equation}\label{eq:besovspacetime}
\|f\|_{\widetilde{L^{q}}(0,T;\dot{B}^{m}_{r,\sigma})}:=\bigg(\sum\limits_{k\in\mathbb{Z}}\big(2^{mk}\|\triangle_{k}f\|_{L^{q}(0,T;L^{r})}\big)^{\sigma}\bigg)^{\frac{1}{\sigma}}.
\end{equation}
We have
\begin{equation}\label{892}
\|f\|_{\widetilde{L^{q}}(0,T;\dot{B}^{m}_{r,\sigma})}\leq \|f\|_{L^{q}(0,T;\dot{B}^{m}_{r,\sigma})}~\mathrm{if}~q\leq\sigma,~~ \|f\|_{L^{q}(0,T;\dot{B}^{m}_{r,\sigma})}\leq \|f\|_{\widetilde{L^{q}}(0,T;\dot{B}^{m}_{r,\sigma})}~\mathrm{if}~\sigma\leq q.
\end{equation}

\subsection{Study of the propagator}\label{sec2.2}
In this paper, the dispersive phenomenon involved in \eqref{CEL} is closely related to the linear operator defined as:
\begin{equation}\label{8111}
e^{i\frac{t}{\delta}p(D)}f(x):=\frac{1}{(2\pi)^{3}}\int_{\mathbb{R}^{3}}{e^{ix\cdot\xi+i\frac{t}{\delta}p(\xi)}\hat{f}(\xi)}d\xi,~~(t,x)\in\mathbb{R}^{1+3},
\end{equation}
where
\begin{equation}\label{191}
\begin{split}
p(\xi)=a\cdot\Big(\sqrt{|\xi|^2+\nu^2+2\nu\xi_3}\pm\sqrt{|\xi|^2+\nu^2-2\nu\xi_3}\Big),~~a=\mathrm{constant}.
\end{split}
\end{equation}
Noting that the propagators in \eqref{2310} can be written in terms of \eqref{8111}. By the notations in \eqref{235} we have
\begin{equation*}
e^{i\frac{t}{\delta}p(D)}f(x)=\sum\limits_{k\in\mathbb{Z}}e^{i\frac{t}{\delta}p(D)}\triangle_{k}f(x).
\end{equation*}
Now we let $\tilde{\varphi}\in\mathcal{S}(\mathbb{R}^{3})$ be a redial function such that
\begin{equation}\label{198}
0\leq\tilde{\varphi}\leq1,~~\mathrm{supp}\,\tilde{\varphi}\subset\{\xi\in\mathbb{R}^{3}:\frac{1}{2}\leq|\xi|\leq3\},~~\tilde{\varphi}\equiv1~~\mathrm{on}~~\mathcal{C}=\{\xi\in\mathbb{R}^{3}:\frac{3}{4}\leq |\xi|\leq \frac{8}{3}\},
\end{equation}
and define
\begin{equation*}
\begin{split}
&\tilde{\varphi}_{k}(\xi):=\tilde{\varphi}(2^{-k}\xi),~~k\in\mathbb{Z}.\\
\end{split}
\end{equation*}
Using the functions above, we introduce the operators
\begin{equation}\label{891}
\Lambda_{k}(t)f(x)=\frac{1}{(2\pi)^{3}}\int_{\mathbb{R}^{3}}{e^{ix\cdot\xi+i\frac{t}{\delta}p(\xi)}\tilde{\varphi}_{k}^{2}(\xi)\hat{f}(\xi)}d\xi,
\end{equation}
and
\begin{equation}\label{7237}
\tilde{\Lambda}_{k}(t)f(x)=\frac{1}{(2\pi)^{3}}\int_{\mathbb{R}^{3}}{e^{ix\cdot\xi+i\frac{t}{\delta}p(\xi)}\tilde{\varphi}_{k}(\xi)\hat{f}(\xi)}d\xi
\end{equation}
for $k\in\mathbb{Z}$. The operators $\Lambda_k(t)$ and $\tilde{\Lambda}_k(t)$ have nearly the same properties. In the following, we shall mainly focus on the operators $\Lambda_{k}(t)$. $\tilde{\Lambda}_{k}(t)$ are actually used to perform the TT* argument in Proposition \ref{a1}. It is clear for $k\in \mathbb{Z}$ that
\begin{equation}\label{1922}
e^{i\frac{t}{\delta}p(D)}\triangle_{k}f(x)=\Lambda_{k}(t)\triangle_{k}f(x)=\tilde{\Lambda}_{k}(t)\triangle_{k}f(x).
\end{equation}

\begin{proposition}\label{pro1}
Let $\Lambda_{k}(t)$, $\tilde{\Lambda}_{k}(t)$ and the phase function $p(\xi)$ be given by \eqref{891}, \eqref{7237} and \eqref{191}, respectively. Then for any $k\in\mathbb{Z}$ we have
\begin{equation*}
\|(\Lambda_k(t)f,\tilde{\Lambda}_{k}(t)f)\|_{L^{\infty}(\mathbb{R}^3)}\leq \frac{C2^{3k}}{\big(1+\frac{|t|}{\varepsilon\mathcal{M}_k}\big)^{1/2}}\|f\|_{L^1(\mathbb{R}^3)},
\end{equation*}
where
\begin{equation}\label{72317}
\mathcal{M}_k=\begin{cases}
1,~\mathrm{if}~k\geq \log_2{\frac{\nu}{60}};\\
\frac{\nu^3}{2^{3k}},~\mathrm{if}~k< \log_2{\frac{\nu}{60}},
\end{cases}
\end{equation}
and the constant $C$ is independent of $\varepsilon$, $\delta$, $f$, $k$ and $t$.
\end{proposition}

The proofs of Proposition \ref{pro1} are presented in Section \ref{sec927}. Now we derive the space-time estimates for the operator $e^{i\frac{t}{\delta}p(D)}$ through Proposition \ref{pro1}. We denote
\begin{equation*}
\|f\|_{L^{q}_{t}L^{r}_{x}}:=\|f\|_{L^{q}(\mathbb{R};L^{r}(\mathbb{R}^{3}))}.
\end{equation*}

\begin{proposition}\label{p2}
Let $\Lambda_{k}(t)$, $\tilde{\Lambda}_{k}(t)$ and the phase function $p(\xi)$ be given by \eqref{891}, \eqref{7237} and \eqref{191}, respectively. Let $r$, $\tilde{r}$, $q$, $\tilde{q}$ satisfy
\begin{equation}\label{8622}
2\leq r,\,\tilde{r}\leq+\infty,~~4\leq q \leq+\infty,~~\frac{2}{q}\leq\frac{1}{2}\bigg(1-\frac{2}{r}\bigg),~~\mathrm{and}~~\frac{2}{\tilde{q}}=\frac{1}{2}\bigg(1-\frac{2}{\tilde{r}}\bigg).
\end{equation}
Then there exist positive constants $C$ such that for any $\varepsilon>0$, $\delta>0$, $f\in L^{2}(\mathbb{R}^{3})$, $F\in L^{\tilde{q}'}(\mathbb{R};L^{\tilde{r}'}(\mathbb{R}^{3}))$ and $k\in\mathbb{Z}$,
\begin{equation}\label{1910}
\|\Lambda_{k}(t)f\|_{L^{q}_{t}L^{r}_{x}}\leq C2^{3k(\frac{1}{2}-\frac{1}{r})}(\mathcal{M}_k\varepsilon)^{\frac{1}{q}}\|f\|_{L^{2}(\mathbb{R}^{3})},
\end{equation}
\begin{equation}\label{1911}
\bigg\|\int^{t}_{-\infty}{\Lambda_{k}(t-s) F(s)}ds\bigg\|_{L^{q}_{t}L^{r}_{x}}\leq C2^{3k(\frac{1}{2}+\frac{2}{\tilde{q}}-\frac{1}{r})}(\mathcal{M}_k\varepsilon)^{\frac{1}{q}+\frac{1}{\tilde{q}}}\| F\|_{L^{\tilde{q}'}_{t}L^{\tilde{r}'}_{x}},
\end{equation}
where $1/\tilde{r}+1/\tilde{r}'=1$, $1/\tilde{q}+1/\tilde{q}'=1$, $\mathcal{M}_k$ is defined by \eqref{72317}, and the constants $C$ are independent of $\varepsilon$, $\delta$, $f$, $F$ and $k$.
\end{proposition}

\begin{proof}
Estimates \eqref{1910} and \eqref{1911} can be proved by combining Proposition \ref{pro1} and Proposition \ref{a1} with
\begin{equation*}
V(t)=\Lambda_k(t),~\mu_1=2^{3k},~\mu_2=(\varepsilon\mathcal{M}_k)^{-1},~\sigma=\frac{1}{2},~\psi=\tilde{\varphi}_k,~\mathrm{and}~d=3.
\end{equation*}
Noting that $\|\mathcal{F}^{-1}(\tilde{\varphi}^2_k)\|_{L^1(\mathbb{R}^3)}\leq C$ for some constant $C$ independent of $k$, we obtain \eqref{1910} by \eqref{862} in Proposition \ref{a1}. On the other hand, since
\begin{equation*}
\|\mathcal{F}^{-1}(\tilde{\varphi}^2_k)\|_{L^{(\frac{1}{2}+\frac{1}{r})^{-1}}(\mathbb{R}^3)}\leq C2^{3k(\frac{1}{2}-\frac{1}{r})}\|\mathcal{F}^{-1}(\tilde{\varphi}^2)\|_{L^{(\frac{1}{2}+\frac{1}{r})^{-1}}(\mathbb{R}^3)}\leq C2^{3k(\frac{1}{2}-\frac{1}{r})},
\end{equation*}
we can also justify \eqref{1911} by using \eqref{865} in Proposition \ref{a1}.
\end{proof}

By using \eqref{892}, \eqref{1922}, the imbedding $\dot{B}^{0}_{\infty,1}\hookrightarrow L^{\infty}$, and Proposition \ref{p2} with $q\in[4,\infty)$, $r=\infty$, $\tilde{q}=\infty$ and $\tilde{r}=2$ we get, for $\forall\,l\in\mathbb{N}$,
\begin{equation}\label{8131}
\begin{split}
&\|\nabla^l e^{i\frac{t}{\delta}p(D)}f\|_{L^q_t L^{\infty}_x}\leq C\|e^{i\frac{t}{\delta}p(D)}\nabla^l f\|_{L^{q}(\mathbb{R};\dot{B}^{0}_{\infty,1})}\leq C \sum\limits_{k\in\mathbb{Z}}\|\Lambda_k(t)\triangle_{k}\nabla^l f\|_{L^{q}_t L^{\infty}_{x}}\\
&~~~~\leq C\varepsilon^{\frac{1}{q}}\bigg(\nu^{\frac{3}{q}}\sum\limits_{k<\log_2{\frac{\nu}{60}}}2^{(\frac{3}{2}-\frac{3}{q})k}\|\triangle_{k}\nabla^lf\|_{L^{2}}+\sum\limits_{k\geq\log_2{\frac{\nu}{60}}}2^{\frac{3}{2}k}\|\triangle_{k}\nabla^lf\|_{L^{2}}\bigg)\\
&~~~~\leq C\varepsilon^{\frac{1}{q}}\Big(\|f\|_{\dot{B}_{2,1}^{\frac{3}{2}-\frac{3}{q}+l}}+\|f\|_{\dot{B}_{2,1}^{\frac{3}{2}+l}}\Big),
\end{split}
\end{equation}
and
\begin{equation}\label{8132}
\begin{split}
&\bigg\|\nabla^l\int_{-\infty}^{t}e^{ i\frac{t-s}{\delta}p(D)}F(s)ds\bigg\|_{L^q_t L^{\infty}_x}\leq C\bigg\|\int_{-\infty}^{t}e^{ i\frac{t-s}{\delta}p(D)}\nabla^l F(s)ds\bigg\|_{L^{q}(\mathbb{R};\dot{B}^{0}_{\infty,1})}\\
&~~~~\leq C\sum\limits_{k\in\mathbb{Z}}\bigg\|\int_{-\infty}^t{\Lambda_k(t-s)\triangle_{k}\nabla^lF(s)}ds\bigg\|_{L^{q}_t L^{\infty}_{x}}\\
&~~~~\leq C\varepsilon^{\frac{1}{q}}\bigg(\nu^{\frac{3}{q}}\sum\limits_{k<\log_2{\frac{\nu}{60}}}2^{(\frac{3}{2}-\frac{3}{q})k}\|\triangle_{k}\nabla^lF\|_{L_t^{1}L^{2}_x}+\sum\limits_{k\geq\log_2{\frac{\nu}{60}}}2^{\frac{3}{2}k}\|\triangle_{k}\nabla^lF\|_{L_t^{1}L^{2}_x}\bigg)\\
&~~~~\leq C\varepsilon^{\frac{1}{q}}\Big( \|F\|_{L^{1}(\mathbb{R};\dot{B}^{\frac{3}{2}-\frac{3}{q}+l}_{2,1})}+\|F\|_{L^{1}(\mathbb{R};\dot{B}^{\frac{3}{2}+l}_{2,1})}\Big)
\end{split}
\end{equation}
for any $q\in[4,\infty)$ and $l\in\mathbb{N}$. Thus, by the imbeddings $H^{2+l}\hookrightarrow B^{\frac{3}{2}+l}_{2,1}\hookrightarrow \dot{B}^{\frac{3}{2}+l}_{2,1}\cap\dot{B}^{\frac{3}{2}-\frac{3}{q}+l}_{2,1}$ we get
\begin{equation}\label{893}
\|\nabla^l e^{i\frac{t}{\delta}p(D)}f\|_{L^q_t L^{\infty}_x}\leq C\varepsilon^{\frac{1}{q}}\|f\|_{H^{2+l}(\mathbb{R}^3)},
\end{equation}
and
\begin{equation}\label{894}
\bigg\|\nabla^l\int_{-\infty}^{t}e^{ i\frac{t-s}{\delta}p(D)}F(s)ds\bigg\|_{L^q_t L^{\infty}_x} \leq C\varepsilon^{\frac{1}{q}}\|F\|_{L^1(\mathbb{R};H^{2+l}(\mathbb{R}^3))}.
\end{equation}

\subsection{Dispersion and longtime existence of the solution}\label{sec2.3}
Now we are ready to prove Theorem \ref{th-main}. From \eqref{2310}, \eqref{893} and \eqref{894} we have, for $l\in\{0,1\}$, $t\in(0,T^*]$ and $q\in[4,\infty)$,
\begin{equation}\label{8121}
\begin{split}
&\|\nabla^l U\|_{L^q(0,t;L^{\infty}(\mathbb{R}^3))}\\
&\leq C\sum\limits_{j=1}^4\bigg(\|\nabla^l e^{-i\frac{\bar{\gamma}\tau}{\delta}p_{j}(D)}\mathcal{P}_{(j)}U_{0}\|_{L^q(0,t;L^{\infty}(\mathbb{R}^3))}\\
&~~~~+\|\nabla^l \int_{0}^{\tau}{e^{-i\bar{\gamma}\frac{\tau-s}{\delta}p_{j}(D)}\mathcal{P}_{(j)}\mathcal{N}(U,U)(s)}ds\|_{L^q(0,t;L^{\infty}(\mathbb{R}^3))}\bigg)\\
&\leq C\varepsilon^{\frac{1}{q}}\big(\|U_0\|_{H^{3}(\mathbb{R}^3)}+\|\mathcal{N}(U,U)\|_{L^1(0,t;H^{3}(\mathbb{R}^3))}\big)\\
&\leq C\varepsilon^{\frac{1}{q}}\big(\|U_0\|_{H^{3}(\mathbb{R}^3)}+\|U\|^2_{L^2(0,t;H^{4}(\mathbb{R}^3))}\big)\\
&\leq C\varepsilon^{\frac{1}{q}}\big(\|U_0\|_{H^{3}(\mathbb{R}^3)}+t\|U\|^2_{L^{\infty}(0,t;H^{4}(\mathbb{R}^3))}\big),
\end{split}
\end{equation}
which implies \eqref{8124} by using \eqref{est} in Proposition \ref{pro-ex}. Then, combining standard estimates for symmetric hyperbolic system in \cite{M-book} and \eqref{8121} with $l=1$ and $q=4$ yield that, for any $m\geq4$ and $t\in(0,T^*]$,
\begin{equation*}
\begin{split}
&\|U\|_{L^{\infty}(0,t;H^{m}(\mathbb{R}^3))}^2\leq \|U_0\|_{H^{m}(\mathbb{R}^3)}^2+C\int^t_0{\|\nabla U\|_{L^{\infty}(\mathbb{R}^3)}\| U\|_{H^{m}(\mathbb{R}^3)}^2}ds\\
&~~~~\leq \|U_0\|_{H^{m}(\mathbb{R}^3)}^2+C_*t^{\frac{3}{4}}\varepsilon^{\frac{1}{4}}\big(\|U_0\|_{H^{m}(\mathbb{R}^3)}+t\|U\|^2_{L^{\infty}(0,t;H^{m}(\mathbb{R}^3))}\big)\| U\|_{L^{\infty}(0,t;H^{m}(\mathbb{R}^3))}^2
\end{split}
\end{equation*}
for some constant $C_*$ which is independent of $\varepsilon$ and $\delta$. We take
\begin{equation*}
\mathcal{T}:=\sup\limits_{t}\big\{t\in[0,T^*]: \|U\|_{L^{\infty}(0,t;H^{m}(\mathbb{R}^3))}\leq 2\|U_0\|_{H^{m}(\mathbb{R}^3)}\big\}.
\end{equation*}
Then by the continuity of the solution we have $\mathcal{T}>0$. Thus for $t\in(0,\mathcal{T}]$ we have
\begin{equation*}
\begin{split}
&\|U\|_{L^{\infty}(0,t;H^{m}(\mathbb{R}^3))}^2\leq \|U_0\|_{H^{m}(\mathbb{R}^3)}^2+C_*t^{\frac{3}{4}}\varepsilon^{\frac{1}{4}}\big(\|U_0\|_{H^{m}(\mathbb{R}^3)}+4t\|U_0\|_{H^{m}(\mathbb{R}^3)}^2\big)\| U\|_{L^{\infty}(0,t;H^{m}(\mathbb{R}^3))}^2.
\end{split}
\end{equation*}
For any $\varepsilon>0$, taking $T_0$ such that
\begin{equation}\label{8122}
C_*T_0^{\frac{3}{4}}\varepsilon^{\frac{1}{4}}\big(\|U_0\|_{H^{m}(\mathbb{R}^3)}+4T_0\|U_0\|_{H^{m}(\mathbb{R}^3)}^2\big)=\frac{1}{2},
\end{equation}
then
\begin{equation*}
\begin{split}
&\|U\|_{L^{\infty}(0,T_0;H^{m}(\mathbb{R}^3))}\leq \sqrt{2}\|U_0\|_{H^{m}(\mathbb{R}^3)}<2\|U_0\|_{H^{m}(\mathbb{R}^3)},
\end{split}
\end{equation*}
which implies $T^*\geq \mathcal{T}\geq T_0$ by the definition of $\mathcal{T}$. On the other hand, let $\varepsilon_0>0$ such that
\begin{equation*}
C_*\varepsilon_0^{1/4}(\|U_0\|_{H^{m}(\mathbb{R}^3)}+4\|U_0\|_{H^{m}(\mathbb{R}^3)}^2)=\frac{1}{4},
\end{equation*}
then when $0<\varepsilon\leq \varepsilon_0$ we get from \eqref{8122} that $T_0\geq1$. Thus
\begin{equation*}
\frac{1}{2}\leq C_*T_0^{\frac{3}{4}}\varepsilon^{\frac{1}{4}}\big(T_0\|U_0\|_{H^{m}(\mathbb{R}^3)}+4T_0\|U_0\|_{H^{m}(\mathbb{R}^3)}^2\big),
\end{equation*}
which implies that
\begin{equation*}
T_0\geq\bigg(\frac{1}{2C_*\big(\|U_0\|_{H^{m}(\mathbb{R}^3)}+4\|U_0\|^2_{H^{m}(\mathbb{R}^3)}\big)}\bigg)^{4/7}\frac{1}{\varepsilon^{1/7}}
\end{equation*}
as $\varepsilon\in(0, \varepsilon_0]$. Thus \eqref{8125} is proved. Now we have completed the proof of Theorem \ref{th-main}.

\section{Proof of Proposition \ref{pro1}}\label{sec927}
This section is devoted specifically to proving Proposition \ref{pro1}. As the dispersive decay estimates for the propagator are designed to be uniform to $\nu$ and cover all frequencies, the constructions of which are complicated. First, by using the convolution inequality and a stretching on the variable $\xi$, it suffices to show that
\begin{equation*}
\sup\limits_{x\in\mathbb{R}^{3}}\bigg|\int_{\mathbb{R}^{3}}{e^{ix\cdot\xi+ i\frac{t}{\delta}p(\xi)}\tilde{\varphi}_{k}(\xi)}d\xi\bigg|=\sup\limits_{x\in\mathbb{R}^{3}}\bigg|2^{3k}\int_{\mathbb{R}^{3}}{e^{ix\cdot2^k\xi+ ia\cdot2^k\frac{t}{\delta}q(\xi)}\tilde{\varphi}(\xi)}d\xi\bigg|\leq \frac{C2^{3k}}{\big(1+\frac{|t|}{\varepsilon\mathcal{M}_k}\big)^{1/2}},
\end{equation*}
where
\begin{equation}\label{7303}
\begin{split}
&q=A\pm B,~~\mathrm{with}\\
&A=\sqrt{|\xi|^2+\sigma_k^2+2\sigma_k\xi_3}=\sqrt{|\xi_h|^2+(\xi_3+\sigma_k)^2},\\
&B=\sqrt{|\xi|^2+\sigma_k^2-2\sigma_k\xi_3}=\sqrt{|\xi_h|^2+(\xi_3-\sigma_k)^2},~~\mathrm{and}~~\sigma_k=\frac{\nu}{2^k}.\\
\end{split}
\end{equation}
As the $L^{\infty}$ norm is invariant under dilation, it suffices to prove that
\begin{equation}\label{7231}
\sup\limits_{x\in\mathbb{R}^{3}}|I(x)|\leq \frac{C2^{3k}}{\big(1+\frac{|t|}{\varepsilon\mathcal{M}_k}\big)^{1/2}},
\end{equation}
where
\begin{equation}\label{7302}
I(x)=2^{3k}\int_{\mathbb{R}^{3}}{e^{i\theta_k(x\cdot\xi+q(\xi))}\tilde{\varphi}(\xi)}d\xi,~~~\theta_k=a\cdot2^k\frac{t}{\delta}.
\end{equation}
We break the integral $I(x)$ into two parts. Let $\tilde{\psi}_{j}\in C_{0}^{\infty}(\mathbb{R}^3)$, $j=1,2$, such that
\begin{equation*}
\begin{split}
&\mathrm{supp}\,\tilde{\psi}_1\subset\Big\{|\xi|\leq 4,~|\xi_3|\leq \frac{1}{28}\Big\},~~\tilde{\psi}_1=1~~\mathrm{on}~~\Big\{|\xi_3|\leq \frac{1}{29}\Big\}\cap\mathrm{supp}\,\tilde{\varphi},\\
&\mathrm{supp}\,\tilde{\psi}_2\subset\Big\{|\xi|\leq4,~|\xi_3|\geq \frac{1}{29}\Big\},~~\tilde{\psi}_2=1~~\mathrm{on}~~\Big\{|\xi_3|\geq \frac{1}{28}\Big\}\cap\mathrm{supp}\,\tilde{\varphi},\\
&\tilde{\psi}_j\geq 0,~~|\nabla_\xi\tilde{\psi}_j|\leq C,~~\sum\nolimits_{j=1}^{2}\tilde{\psi}_j(\xi)=1~\mathrm{on~supp}\,\tilde{\varphi},~\mathrm{and}~\tilde{\psi}_2~\mathrm{is~a~radial~function~of}~\xi_h.\\
\end{split}
\end{equation*}
Thus, by denoting
\begin{equation*}
\psi_j:=\tilde{\varphi}\tilde{\psi}_j,~~j=1,2,~~\mathrm{and}~~\phi(\xi)=\phi(\xi_1,\xi_2,\xi_3):=x\cdot\xi+q(\xi),
\end{equation*}
we can write
\begin{equation*}
I(x)=I_1(x)+I_2(x),~~\mathrm{where}~~I_j(x)=2^{3k}\int_{\mathbb{R}^{3}}{e^{i\theta_k\phi(\xi)}\psi_j(\xi)}d\xi.
\end{equation*}
The estimates for $I_{1}(x)$ and $I_2(x)$ are presented in Section \ref{sec3.1} and Section \ref{sec3.2}, respectively. The discussions in each subsections are divided into three parts:
\begin{equation*}
\begin{split}
&\mathrm{high~frequencies}:~~k\geq\log_{2}{60\nu},~~\mathrm{or}~~\sigma_k\leq\frac{1}{60};\\
&\mathrm{middle~frequencies}:~~\log_{2}{\frac{\nu}{60}}\leq k< \log_{2}{60\nu},~~\mathrm{or}~~\frac{1}{60}<\sigma_k\leq 60;\\
&\mathrm{low~frequencies}:~~k<\log_{2}{\frac{\nu}{60}},~~\mathrm{or}~~\sigma_k>60.\\
\end{split}
\end{equation*}
For $q=A\pm B$ we can calculate that
\begin{equation}\label{7238}
\begin{split}
&\partial_{\xi_1}q=\Big(\frac{1}{A}\pm\frac{1}{B}\Big)\xi_1,~~~\partial_{\xi_2}q=\Big(\frac{1}{A}\pm\frac{1}{B}\Big)\xi_2,~~~\partial_{\xi_3}q=\frac{\xi_3+\sigma_k}{A}\pm\frac{\xi_3-\sigma_k}{B},\\
&\partial_{\xi_1}^2 q=\Big(\frac{1}{A}\pm\frac{1}{B}\Big)-\Big(\frac{1}{A^3}\pm\frac{1}{B^3}\Big)\xi_1^2,~~~\partial_{\xi_2}^2 q=\Big(\frac{1}{A}\pm\frac{1}{B}\Big)-\Big(\frac{1}{A^3}\pm\frac{1}{B^3}\Big)\xi_2^2,\\
&\partial_{\xi_1\xi_2}q=-\Big(\frac{1}{A^3}\pm\frac{1}{B^3}\Big)\xi_1\xi_2,~~~\partial_{\xi_3}^2 q=\Big(\frac{1}{A}\pm\frac{1}{B}\Big)-\Big(\frac{(\xi_3+\sigma_k)^2}{A^3}\pm\frac{(\xi_3-\sigma_k)^2}{B^3}\Big),\\
&\partial_{\xi_1\xi_3}q=-\Big(\frac{\xi_3+\sigma_k}{A^3}\pm\frac{\xi_3-\sigma_k}{B^3}\Big)\xi_1,~~~\partial_{\xi_2\xi_3}q=-\Big(\frac{\xi_3+\sigma_k}{A^3}\pm\frac{\xi_3-\sigma_k}{B^3}\Big)\xi_2.
\end{split}
\end{equation}
These derivatives shall be frequently used.

\textbf{Notations.} In the following, for $\xi=(\xi_1,\xi_2,\xi_3)$ and $i,j=1,2,3$, we denote by:
\begin{equation*}
\xi_{i,j}:=(\xi_i,\xi_j),~~d\xi_{i,j}:=d\xi_id\xi_j,~~\nabla_{\xi_{i,j}}:=(\partial_{\xi_i},\partial_{\xi_j}),~~\mathrm{and}~~D_{\xi_{i,j}}^2 :=\left(\begin{matrix}
\partial_{\xi_i}^2 &\partial_{\xi_i\xi_j} \\
\partial_{\xi_i\xi_j} &\partial_{\xi_j}^2
\end{matrix}\right).
\end{equation*}

\subsection{Estimates for $I_1$}\label{sec3.1}
The integral $I_1$ should be further broken. We can see that
\begin{equation*}
\mathrm{supp}\,\psi_{1}\subset \Big\{\frac{1}{2}\leq|\xi|\leq 3,~|\xi_h|\geq \frac{\sqrt{3}}{4},~|\xi_3|\leq \frac{1}{28}\Big\}.
\end{equation*}
Let $\psi_{1,1}=\psi_{1,1}(\xi_1)\in C_{0}^{\infty}(\mathbb{R})$ such that
\begin{equation*}
\begin{split}
&0\leq\psi_{1,1}\leq 1,~~\mathrm{supp}\,\psi_{1,1}\subset\Big\{|\xi_1|\leq \frac{\sqrt{2}}{4}\Big\},~~\psi_{1,1}=1~~\mathrm{on}~~\Big\{|\xi_1|\leq\frac{1}{4}\sqrt{\frac{3}{2}}\Big\},~~\mathrm{and}~~|\partial_{\xi_1}\psi_{1,1}|\leq C,\\
\end{split}
\end{equation*}
then we can break $I_{1}(x)=I_{1,1}(x)+I_{1,2}(x)$ where
\begin{equation*}
I_{1,1}(x)=2^{3k}\int_{\mathbb{R}^{3}}{e^{i\theta_k\phi(\xi)}\psi_1(\xi)\psi_{1,1}(\xi_1)}d\xi,
\end{equation*}
and
\begin{equation*}
I_{1,2}(x)=2^{3k}\int_{\mathbb{R}^{3}}{e^{i\theta_k\phi(\xi)}\psi_1(\xi)(1-\psi_{1,1}(\xi_1))}d\xi.
\end{equation*}
In the following we pay our attentions to $I_{1,1}$, since the estimates for $I_{1,2}$ are similar. Define
\begin{equation*}
\Xi_{1}(\xi):=\psi_1(\xi)\psi_{1,1}(\xi_1),
\end{equation*}
then we have
\begin{equation*}
\begin{split}
&I_{1,1}(x)=2^{3k}\int_{\mathbb{R}^{3}}{e^{i\theta_k\phi(\xi)}\Xi_1(\xi)}d\xi,~~\mathrm{and}\\
&\mathrm{supp}\,\Xi_{1}\subset\Big\{\frac{1}{2}\leq|\xi|\leq 3,~|\xi_h|\geq \frac{\sqrt{3}}{4},~|\xi_1|\leq \frac{\sqrt{2}}{4},~|\xi_2|\geq \frac{1}{4},~|\xi_3|\leq \frac{1}{28}\Big\}.
\end{split}
\end{equation*}
By using a partition of unity to cover the support of the above $\Xi_{1}$, we may assume that the support of $\Xi_{1}$ is sufficiently small such that for any $\xi,\eta\in\mathrm{supp}\,\Xi_{1}$, the line segment connecting $\xi$ and $\eta$ lies wholly in
\begin{equation*}
\Omega_{1}:=\Big\{\frac{1}{4}\leq|\xi|\leq 4,~|\xi_h|\geq \frac{\sqrt{2}}{4},~|\xi_1|\leq \frac{\sqrt{3}}{4},~|\xi_2|\geq \frac{1}{8},~|\xi_3|\leq \frac{1}{26}\Big\}.
\end{equation*}

\subsubsection{Estimates for high frequencies}
In this subsection we construct the estimates for $I_{1,1}$ with high frequencies such that $k\geq\log_{2}{60\nu}$, which is equivalent to $\sigma_k\leq \frac{1}{60}$. The methods designed for $q=A-B$ and $q=A+B$ are different.

\textbf{The~case~of~$q=A-B$.} In this case the decay estimate is derived from the oscillation integral on $\mathbb{R}^2_{\xi_2,\xi_3}$. We let
\begin{equation*}
J_{1}=J_{1}(x,\xi_1,\xi_3):=\int{e^{i\theta_k\phi(\xi)}\Xi_{1}(\xi)}d\xi_{2},
\end{equation*}
then
\begin{equation}\label{e9112}
I_{1,1}=2^{3k}\iint{J_{1}}d\xi_{1,3}\leq C2^{3k}\Big(\iint{|J_{1}|^2}d\xi_{1,3}\Big)^{1/2},
\end{equation}
and
\begin{equation}\label{7236}
\begin{split}
|J_{1}|^2=J_{1}\bar{J}_{1}&=\iint{e^{i\theta_k(\phi(\xi_1,\eta_2,\xi_3)-\phi(\xi))}\Xi_{1}(\xi_1,\eta_2,\xi_3)\bar{\Xi}_{1}(\xi)}d\xi_{2}d\eta_{2}\\
&=\iint{e^{i\theta_k(\phi(\xi_1,\xi_2+\tau_2,\xi_3)-\phi(\xi))}\Upsilon_1(\xi,\tau_2)}d\xi_{2}d\tau_{2},
\end{split}
\end{equation}
where
\begin{equation*}
\Upsilon_1(\xi,\tau_2)=\Xi_{1}(\xi_1,\xi_2+\tau_2,\xi_3)\bar{\Xi}_{1}(\xi).
\end{equation*}
Noting that, in \eqref{7236} the variable $\tau_2=\eta_2-\xi_2\in[-8,8]$ since $\eta_2,\xi_2\in[-4,4]$. We introduce the operator $L_{1}$ by
\begin{equation*}
\begin{split}
L_{1} f(\xi)=\frac{1}{i\theta_k}g_{1}\cdot\partial_{\xi_3} f(\xi),~\mathrm{where}~g_1=\frac{1}{\partial_{\xi_3}\phi(\xi_1,\xi_2+\tau_2,\xi_3)-\partial_{\xi_3}\phi(\xi)}.
\end{split}
\end{equation*}
Then $L_{1}^{\top}f=-\frac{1}{i\theta_k}\partial_{\xi_3}(g_{1} f)$ and
\begin{equation}\label{7211}
\begin{split}
\iint{|J_{1}|^2}d\xi_{1,3}&=\iint{(L_1^N)\Big(e^{i\theta_k(\phi(\xi_1,\xi_2+\tau_2,\xi_3)-\phi(\xi))}\Big)\Upsilon_1(\xi,\tau_2)}d\xi_2d\tau_2d\xi_{1,3}\\
&=\iint{e^{i\theta_k(\phi(\xi_1,\xi_2+\tau_2,\xi_3)-\phi(\xi))}(L_1^{\top})^N\Upsilon_1(\xi,\tau_2)}d\xi_2d\tau_2d\xi_{1,3}\\
&\leq C\iint{|(L_1^{\top})^N\Upsilon_1(\xi,\tau_2)|}d\xi_2d\tau_2d\xi_{1,3}\\
\end{split}
\end{equation}
for any $N\in\mathbb{N}$. Taking $N=2$ in \eqref{7211} we get
\begin{equation}\label{8231}
\begin{split}
\iint{|J_{1}|^2}d\xi_{1,3}\leq C\iint{\frac{1}{|\theta_k|^2}\sum\limits_{j_1+j_2=0}^{2}|\partial_{\xi_3}^{j_1}g_{1}|\cdot|\partial_{\xi_3}^{j_2}g_{1}|}d\xi_2d\tau_2d\xi_{1,3}.
\end{split}
\end{equation}
By the differential mean value theorem we have, in the above integral,
\begin{equation}\label{921}
|g_{1}|\leq\frac{C}{\inf_{\xi\in\Omega_{1}}|\partial_{\xi_2\xi_3}q(\xi)|\cdot|\tau_{2}|},~~~|\partial_{\xi_3} g_{1}|\leq \frac{C\sup_{\xi\in\Omega_{1}}|\partial_{\xi_2}\partial_{\xi_3}^2q(\xi)|}{\inf_{\xi\in\Omega_{1}}|\partial_{\xi_2\xi_3}q(\xi)|^2\cdot|\tau_{2}|},
\end{equation}
and
\begin{equation}\label{922}
\begin{split}
|\partial_{\xi_3}^{2} g_{1}|&\leq\frac{C}{|\tau_{2}|}\bigg( \frac{\sup_{\xi\in\Omega_{1}}|\partial_{\xi_2}\partial_{\xi_3}^3q(\xi)|}{\inf_{\xi\in\Omega_{1}}|\partial_{\xi_2\xi_3}q(\xi)|^2}+\frac{\sup_{\xi\in\Omega_{1}}|\partial_{\xi_2}\partial_{\xi_3}^2q(\xi)|^{2}}{\inf_{\xi\in\Omega_{1}}|\partial_{\xi_2\xi_3}q(\xi)|^3}\bigg).
\end{split}
\end{equation}
Thus it suffices to estimate the upper bounds of $|\partial_{\xi_2}\partial_{\xi_3}^2q(\xi)|$, $|\partial_{\xi_2}\partial_{\xi_3}^3q(\xi)|$ and the lower bound of $|\partial_{\xi_2\xi_3}q(\xi)|$ on $\Omega_{1}$.

From \eqref{7238} we have, for $q=A-B$,
\begin{equation*}
\begin{split}
\partial_{\xi_2\xi_3}q=-\Big(\frac{\xi_3+\sigma_k}{A^3}-\frac{\xi_3-\sigma_k}{B^3}\Big)\xi_2=-\frac{\xi_2}{A^3B^3}\Big((A^3+B^3)\sigma_k-(A^3-B^3)\xi_3\Big).
\end{split}
\end{equation*}
Noting that, for $\xi\in\Omega_{1}$,
\begin{equation*}
(A^3+B^3)\sigma_k\geq 2|\xi_h|^3\sigma_k\geq \frac{\sqrt{2}}{16}\sigma_k,
\end{equation*}
and
\begin{equation*}
|(A^3-B^3)\xi_3|=\Big|(A^2-B^2)\frac{A^2+B^2+AB}{A+B}\xi_3\Big|\leq |4\sigma_k\xi_3^2(A+B)|\leq 40\xi^2_3\sigma_k\leq \frac{10}{169}\sigma_k
\end{equation*}
by observing that $A,B\leq 5$ as $\sigma_k\leq \frac{1}{60}$. Thus
\begin{equation}\label{e9131}
|\partial_{\xi_2\xi_3}q|\geq \frac{C|\xi_2|\sigma_k}{A^3B^3}\geq C\sigma_k.
\end{equation}
Furthermore, noting that $\partial_{\xi_2}\partial_{\xi_3}^2q(\xi)$ and $\partial_{\xi_2}\partial_{\xi_3}^3q(\xi)$ are the linear combinations of terms having the following forms:
\begin{equation*}
\Big(\frac{1}{A^n}-\frac{1}{B^n}\Big)\xi_2^{j}\xi_3^l,~~\Big(\frac{1}{A^n}-\frac{1}{B^n}\Big)\xi_2^{j}\xi_3^l\sigma_k^{m},~~\Big(\frac{1}{A^n}+\frac{1}{B^n}\Big)\xi_2^{j}\xi_3^l\sigma_k^m,
\end{equation*}
where $m,n,j,l$ are integers such that $m,n\geq1$ and $j,l\geq 0$. We can verify that
\begin{equation*}
\Big|\Big(\frac{1}{A^n}-\frac{1}{B^n}\Big)\xi_2^{j}\xi_3^l\Big|\leq C\Big|\frac{1}{A}-\frac{1}{B}\Big|\leq \frac{C\sigma_k}{AB(A+B)}\leq C\sigma_k,
\end{equation*}
and
\begin{equation*}
\Big|\Big(\frac{1}{A^n}\pm\frac{1}{B^n}\Big)\xi_2^{j}\xi_3^l\sigma_k^m\Big|\leq C\sigma_k^m\leq C\sigma_k.
\end{equation*}
Thus $|\partial_{\xi_2}\partial_{\xi_3}^2q(\xi)|,|\partial_{\xi_2}\partial_{\xi_3}^3q(\xi)|\leq C\sigma_k$. Introducing the above estimates into \eqref{921} and \eqref{922}, and plugging the results into \eqref{8231} yields
\begin{equation}\label{911}
\begin{split}
\iint{|J_{1}|^2}d\xi_{1,3}\leq C\iint{\Big(\frac{1}{\theta_k \sigma_k\tau_2}\Big)^2}d\xi_2d\tau_2d\xi_{1,3}\leq C\int_{-8}^{8}{\frac{1}{|\theta_k \sigma_k\tau_2|^2}}d\tau_2.
\end{split}
\end{equation}
On the other hand, taking $N=0$ in \eqref{7211} we get
\begin{equation}\label{e9111}
 \iint{|J_{1}|^2}d\xi_{1,3}\leq C\int_{-8}^{8}{1}d\tau_2.
\end{equation}
Thus, combining \eqref{911} and \eqref{e9111} we obtain
\begin{equation}\label{913}
\iint{|J_{1}|^2}d\xi_{1,3}\leq C\int_{-8}^{8}{\frac{1}{1+|\theta_k \sigma_k\tau_2|^2}}d\tau_2\leq \frac{C}{|\theta_k\sigma_k|}.
\end{equation}
Introducing \eqref{913} into \eqref{e9112} we finally obtain that
\begin{equation}\label{9116}
|I_{1,1}|\leq C2^{3k}\Big(\frac{1}{|\theta_k\sigma_k|}\Big)^{1/2}\leq C2^{3k}\Big(\frac{\varepsilon}{|t|}\Big)^{1/2}
\end{equation}
for $q=A-B$ and $\sigma_k\leq\frac{1}{60}$.

\begin{remark}
We explain why we only use the mixed derivative $\partial_{\xi_2\xi_3}q$ but not the Hessian matrix of $q$ on $\mathbb{R}^2_{\xi_2,\xi_3}$. Indeed, from \eqref{7238} we have
\begin{equation*}
\mathrm{det}\,D_{\xi_{2,3}}^2 q(\xi)=\Big(\frac{1}{A}-\frac{1}{B}\Big)^2\cdot\Big(\frac{1}{A^2}+\frac{1}{B^2}+\frac{1}{AB}\Big)\xi^2_1-\frac{4\sigma_k^2\xi_2^2}{A^3B^3}.
\end{equation*}
Thus, for $\xi\in\Omega_1$ we can check that $|\mathrm{det}\,D_{\xi_{2,3}}^2 q(\xi)|\geq C\sigma^2_k$, which is smaller than the lower bound of $|\partial_{\xi_2\xi_3}q|$ in \eqref{e9131} as $\sigma_k\ll1$. As a result, we can obtain that
\begin{equation*}
|I_{1,1}|\leq C2^{3k}\Big(\frac{1}{|\theta_k\sigma_k^n|}\Big)^{1/2}\leq C2^{3k}\Big(\frac{2^{(n-1)k}\varepsilon^n}{|t|\delta^{n-1}}\Big)^{1/2}~~\mathrm{for~some}~~n\in[2,\infty)\cap\mathbb{N},
\end{equation*}
which may go to infinity as $\varepsilon,\delta\to0$.
\end{remark}

\textbf{The~case~of~$q=A+B$.} In this case the decay estimate is derived from the oscillation integral on the plane $\mathbb{R}_{\xi_1,\xi_3}^2$. However, the proof can not be obtained by simply switching the roles of $\xi_1$ and $\xi_2$ in the above discussions as we need to use the Hessian matrix on $\mathbb{R}_{\xi_1,\xi_3}^2$ here. We let
\begin{equation*}
J_{2}=J_{2}(x,\xi_2):=\int{e^{i\theta_k\phi(\xi)}\Xi_{1}(\xi)}d\xi_{1,3},
\end{equation*}
then $I_{1,1}=2^{3k}\int{J_{2}}d\xi_2$ and similar to \eqref{7236},
\begin{equation}\label{918}
\begin{split}
|J_{2}|^2=J_{2}\bar{J}_{2}&=\iint{e^{i\theta_k(\phi(\xi_1+\tau_1,\xi_2,\xi_3+\tau_3)-\phi(\xi))}\Upsilon_2(\xi,\tau_1,\tau_3)}d\xi_{1,3}d\tau_{1,3},
\end{split}
\end{equation}
where
\begin{equation*}
\Upsilon_2(\xi,\tau_1,\tau_3)=\Xi(\xi_1+\tau_1,\xi_2,\xi_3+\tau_3)\bar{\Xi}(\xi).
\end{equation*}
We introduce the operator $L_{2}$ by
\begin{equation*}
L_{2} f(\xi)=\frac{1}{i\theta_k}g_{2}\cdot\nabla_{\xi_{1,3}} f(\xi),~~\mathrm{where}~~g_{2}=\frac{\nabla_{\xi_{1,3}}\phi(\xi_1+\tau_1,\xi_2,\xi_3+\tau_3)-\nabla_{\xi_{1,3}}\phi(\xi)}{|\nabla_{\xi_{1,3}}\phi(\xi_1+\tau_1,\xi_2,\xi_3+\tau_3)-\nabla_{\xi_{1,3}}\phi(\xi)|^{2}}.
\end{equation*}
Then $L_{2}^{\top}f=-\frac{1}{i\theta_k}\nabla_{\xi_{1,3}}\cdot(g_{2} f)$ and
\begin{equation}\label{917}
\begin{split}
|J_{2}|^2&=\iint{(L_2)^N\big(e^{i\theta_k(\phi(\xi_1+\tau_1,\xi_2,\xi_3+\tau_3)-\phi(\xi))}\big)\Upsilon_{2}(\xi,\tau_1,\tau_3)}d\xi_{1,3}d\tau_{1,3}\\
&=\iint{e^{i\theta_k(\phi(\xi_1+\tau_1,\xi_2,\xi_3+\tau_3)-\phi(\xi))}(L_{2}^{\top})^N\Upsilon_2(\xi,\tau_1,\tau_3)}d\xi_{1,3}d\tau_{1,3}
\end{split}
\end{equation}
for any $N\in\mathbb{N}$. Taking $N=2$ in the above identity we get, similar to \eqref{8231},
\begin{equation}\label{916}
\begin{split}
|J_2|^2\leq
\iint{\frac{C}{|\theta_k|^2}\sum\limits_{j_1+j_2=0}^{2}|\nabla^{j_1}_{\xi_{1,3}}g_{2}|\cdot|\nabla^{j_2}_{\xi_{1,3}}g_{2}|}d\xi_{1,3}d\tau_{1,3}.
\end{split}
\end{equation}
By the differential mean value theorem we have
\begin{equation}\label{e9121}
|g_{2}|\leq \frac{1}{|\nabla_{\xi_{1,3}} q(\xi_1+\tau_1,\xi_2,\xi_3+\tau_3)-\nabla_{\xi_{1,3}} q(\xi)|}\leq\frac{C}{\inf_{\xi\in\Omega_{1}}|\mathrm{det}\,D_{\xi_{1,3}}^{2}q(\xi)|\cdot|\tau_{1,3}|},
\end{equation}
\begin{equation}\label{e9122}
\begin{split}
|\nabla_{\xi_{1,3}} g_{2}|&\leq \frac{C|\nabla_{\xi_{1,3}}^{2}q(\xi_1+\tau_1,\xi_2,\xi_3+\tau_3)-\nabla_{\xi_{1,3}}^{2}q(\xi)|}{|\nabla_{\xi_{1,3}} q(\xi_1+\tau_1,\xi_2,\xi_3+\tau_3)-\nabla_{\xi_{1,3}} q(\xi)|^{2}}\leq \frac{C\sup_{\xi\in\Omega_{1}}|\nabla_{\xi_{1,3}}^{3}q(\xi)|}{\inf_{\xi\in\Omega_{1}}|\mathrm{det}\,D_{\xi_{1,3}}^{2}q(\xi)|^{2}\cdot|\tau_{1,3}|},
\end{split}
\end{equation}
and
\begin{equation}\label{e9123}
\begin{split}
|\nabla_{\xi_{1,3}}^{2} g_{2}|&\leq\frac{C}{|\tau_{1,3}|}\bigg( \frac{\sup_{\xi\in\Omega_{1}}|\nabla_{\xi_{1,3}}^{4}q(\xi)|}{\inf_{\xi\in\Omega_{1}}|\mathrm{det}\,D_{\xi_{1,3}}^{2}q(\xi)|^{2}}+\frac{\sup_{\xi\in\Omega_{1}}|\nabla_{\xi_{1,3}}^{3}q(\xi)|^{2}}{\inf_{\xi\in\Omega_{1}}|\mathrm{det}\,D_{\xi_{1,3}}^{2}q(\xi)|^{3}}\bigg).
\end{split}
\end{equation}
Thus it suffices to estimate the upper bounds of $|\nabla_{\xi_{1,3}}^{3}q|$, $|\nabla_{\xi_{1,3}}^{4}q|$ and the lower bound of $|\mathrm{det}\,D_{\xi_{1,3}}^{2}q|$ on $\Omega_{1}$. From \eqref{7238} we have, for $q=A+B$ and $\sigma_k\leq \frac{1}{60}$,
\begin{equation*}
\begin{split}
&\mathrm{det}\,D_{\xi_{1,3}}^2 q(\xi)=\Big(\frac{1}{A}+\frac{1}{B}\Big)\cdot\Big(\frac{1}{A^3}+\frac{1}{B^3}\Big)\xi^2_2+\frac{4\sigma_k^2\xi_1^2}{A^3B^3}\geq\Big(\frac{1}{A}+\frac{1}{B}\Big)\cdot\Big(\frac{1}{A^3}+\frac{1}{B^3}\Big)\xi^2_2\geq C,
\end{split}
\end{equation*}
and
\begin{equation*}
|\nabla_{\xi_{1,3}}^{3}q|,~|\nabla_{\xi_{1,3}}^{4}q|\leq C.
\end{equation*}
Introducing the above estimates into \eqref{e9121}, \eqref{e9122} and \eqref{e9123}, and plugging the results into \eqref{916} we get
\begin{equation*}
\begin{split}
|J_{2}|^2\leq \iint{\frac{C}{|\theta_k|^2|\tau_{1,3}|^2}}d\xi_{1,3}d\tau_{1,3}
&\leq C\iint{\frac{1}{|\theta_k|^2|\tau_{1,3}|^3}}d\tau_{1,3},
\end{split}
\end{equation*}
where we have used the fact $|\tau_{1,3}|\leq C$. Taking $N=0$ in \eqref{917} we can also get
\begin{equation*}
|J_2|^2\leq C\iint{1}d\tau_{1,3}.
\end{equation*}
Thus, combining the above two estimates we have
\begin{equation*}
|J_2|\leq \bigg(C\iint{\frac{1}{1+|\theta_k|^{2}|\tau_{1,3}|^3}}d\tau_{1,3}\bigg)^{1/2}\leq C\Big(\frac{1}{|\theta_k|}\Big)^{2/3}.
\end{equation*}
Plugging the above estimate into $I_{1,1}$ yields
\begin{equation}\label{9117}
|I_{1,1}|\leq C2^{3k}\Big(\frac{1}{|\theta_k|}\Big)^{2/3}.
\end{equation}
As $\sigma_k\leq \frac{1}{60}$, we finally obtain that
\begin{equation}\label{924}
|I_{1,1}|\leq C2^{3k}\Big(\frac{1}{|\theta_k|\sigma_k}\Big)^{2/3}\leq C2^{3k}\Big(\frac{\varepsilon}{|t|}\Big)^{2/3}\leq C2^{3k}\Big(\frac{\varepsilon}{|t|}\Big)^{1/2}
\end{equation}
for $q=A+B$.

Combining \eqref{9116} and \eqref{924} we conclude with
\begin{equation}\label{9118}
|I_{1,1}|\leq C2^{3k}\Big(\frac{\varepsilon}{|t|}\Big)^{1/2}
\end{equation}
for $q=A\pm B$ and $\sigma_k\leq\frac{1}{60}$.

\subsubsection{Estimates for middle frequencies}
The decay estimate of $I_{1,1}$ with $\frac{1}{60}< \sigma_k\leq60$ can be derived by switching the roles of $\xi_1$ and $\xi_2$ in the discussions for high frequencies between \eqref{918} and \eqref{924}. Thus, it suffices to estimate the upper bounds of $|\nabla_{\xi_{2,3}}^{3}q|$, $|\nabla_{\xi_{2,3}}^{4}q|$ and the lower bound of $|\mathrm{det}\,D_{\xi_{2,3}}^{2}q|$ on $\Omega_{1}$.
\begin{itemize}
\item For $q=A-B$ we can calculate from \eqref{7238} that
\begin{equation*}
\mathrm{det}\,D_{\xi_{2,3}}^2 q(\xi)=\Big(\frac{1}{A}-\frac{1}{B}\Big)^2\cdot\Big(\frac{1}{A^2}+\frac{1}{B^2}+\frac{1}{AB}\Big)\xi^2_1-\frac{4\sigma_k^2\xi_2^2}{A^3B^3}.
\end{equation*}
For $\xi\in\Omega_1$ and $\frac{1}{60}< \sigma_k\leq60$ we have
\begin{equation*}
\Big|\Big(\frac{1}{A}-\frac{1}{B}\Big)^2\cdot\Big(\frac{1}{A^2}+\frac{1}{B^2}+\frac{1}{AB}\Big)\xi^2_1\Big|\leq\frac{3(A^2-B^2)^2}{A^2B^2(A+B)^2}\leq\frac{3(4\sigma_k\xi_3)^2}{4A^3B^3}\leq\frac{3\sigma^2_k}{169A^3B^3},
\end{equation*}
and
\begin{equation*}
\frac{4\sigma_k^2\xi_2^2}{A^3B^3}\geq\frac{\sigma_k^2}{16A^3B^3}.
\end{equation*}
Thus
\begin{equation*}
\begin{split}
&|\mathrm{det}\,D_{\xi_{2,3}}^2 q|\geq\frac{C\sigma^2_k}{A^3B^3}\geq C.\\
\end{split}
\end{equation*}
\item For $q=A+B$ we have
\begin{equation*}
\begin{split}
&\mathrm{det}\,D_{\xi_{2,3}}^2 q(\xi)=\Big(\frac{1}{A}+\frac{1}{B}\Big)^2\cdot\Big(\frac{1}{A^2}+\frac{1}{B^2}-\frac{1}{AB}\Big)\xi^2_1+\frac{4\sigma_k^2\xi_2^2}{A^3B^3}>\frac{4\sigma_k^2\xi_2^2}{A^3B^3}\geq C.
\end{split}
\end{equation*}
\end{itemize}
Furthermore, since $A,B\geq |\xi_h|\geq C$ on $\Omega_{1}$, we can verify that $|\nabla_{\xi_{2,3}}^{3}q|,|\nabla_{\xi_{2,3}}^{4}q|\leq C$ for both $q=A-B$ and $q=A+B$. Thus, by switching the roles of $\xi_1$ and $\xi_2$ in the discussions \eqref{918}-\eqref{924} we can finally obtain
\begin{equation}\label{925}
|I_{1,1}|\leq C2^{3k}\Big(\frac{1}{|\theta_k|}\Big)^{2/3}\leq C2^{3k}\Big(\frac{1}{|\theta_k|\sigma_k}\Big)^{2/3}\leq C2^{3k}\Big(\frac{\varepsilon}{|t|}\Big)^{2/3}\leq C2^{3k}\Big(\frac{\varepsilon}{|t|}\Big)^{1/2}
\end{equation}
for $q=A\pm B$ and $\frac{1}{60}< \sigma_k\leq60$.

\subsubsection{Estimates for low frequencies}
In this subsection we construct the decay estimates for $I_{1,1}$ with low frequencies such that $k<\log_2 \frac{\nu}{60}$, which is equivalent to $\sigma_k> 60$.

\textbf{The~case~of~$q=A-B$.} In this case the decay estimate is derived from the oscillation integral on $\mathbb{R}^2_{\xi_2,\xi_3}$. The discussions are similar to those for the high frequencies in \eqref{e9112}-\eqref{9116}. We go on from \eqref{922}. Thus it suffices to estimate the upper bounds of $|\partial_{\xi_2}\partial_{\xi_3}^2q(\xi)|$, $|\partial_{\xi_2}\partial_{\xi_3}^3q(\xi)|$ and the lower bound of $|\partial_{\xi_2\xi_3}q(\xi)|$ for $q=A-B$ on $\Omega_{1}$. Noting that, as $\sigma_k>60$,
\begin{equation*}
|(A^3-B^3)\xi_3|\leq\frac{1}{2}\sigma_k (A^3+B^3),~~\mathrm{and}~~\frac{1}{2}\sigma_k\leq A,B\leq 2\sigma_k.
\end{equation*}
Thus
\begin{equation*}
\begin{split}
|\partial_{\xi_2\xi_3}q|&=\Big|\frac{\xi_2}{A^3B^3}\Big((A^3+B^3)\sigma_k-(A^3-B^3)\xi_3\Big)\Big|\geq \frac{C}{A^3B^3}(A^3+B^3)\sigma_k\geq\frac{C}{\sigma_k^2}.
\end{split}
\end{equation*}
Furthermore, we can verify that
\begin{equation*}
|\partial_{\xi_2}\partial_{\xi_3}^2q|\leq C\Big(\Big|\frac{1}{A^3}-\frac{1}{B^3}\Big|+\Big|\frac{(\xi_3+\sigma_k)^2}{A^5}-\frac{(\xi_3-\sigma_k)^2}{B^5}\Big|\Big)\leq C\Big(\frac{1}{A^3}+\frac{1}{B^3}\Big)\leq  \frac{C}{\sigma_k^3},
\end{equation*}
and
\begin{equation*}
\begin{split}
|\partial_{\xi_2}\partial_{\xi_3}^3q|&\leq C\Big(\Big|\frac{\xi_3+\sigma_k}{A^5}-\frac{\xi_3-\sigma_k}{B^5}\Big|+\Big|\frac{(\xi_3+\sigma_k)^3}{A^7}-\frac{(\xi_3-\sigma_k)^3}{B^7}\Big|\Big)\leq C\Big(\frac{1}{A^4}+\frac{1}{B^4}\Big)\leq \frac{C}{\sigma_k^4}.
\end{split}
\end{equation*}
Thus, plugging the above estimates into \eqref{921} and \eqref{922} we have
\begin{equation}\label{926}
|g_1|\leq\frac{C\sigma_k^2}{|\tau_2|},~~|\partial_{\xi_3}g_1|\leq\frac{C\sigma_k}{|\tau_2|},~~\mathrm{and}~~|\partial^2_{\xi_3}g_1|\leq\frac{C}{|\tau_2|}.
\end{equation}
Introducing \eqref{926} into \eqref{8231} we further have
\begin{equation}\label{923}
\begin{split}
\iint{|J_{1}|^2}d\xi_{1,3}\leq C\iint{\Big(\frac{\sigma_k^2}{\theta_k\tau_2}\Big)^{2}}d\xi_2d\tau_2d\xi_{1,3}\leq C\int^8_{-8}{\Big(\frac{\sigma_k^2}{\theta_k\tau_2}\Big)^{2}}d\tau_2.
\end{split}
\end{equation}
Thus, repeating the discussions in \eqref{e9111}, \eqref{913} and \eqref{9116} we can finally obtain that
\begin{equation}\label{928}
|I_{1,1}|\leq C2^{3k}\Big(\frac{\sigma_k^2}{|\theta_k|}\Big)^{1/2}
\end{equation}
for $q=A-B$ and $\sigma_k> 60$.

\textbf{The~case~of~$q=A+B$.} In this case the decay estimate is derived from the oscillation integral on $\mathbb{R}_{\xi_1}$. We let
\begin{equation*}
J_{3}=J_{3}(x,\xi_2,\xi_3):=\int{e^{i\theta_k\phi(\xi)}\Xi_{1}(\xi)}d\xi_{1},
\end{equation*}
then $I_{1,1}=2^{3k}\int{J_{3}}d\xi_2d\xi_3$ and
\begin{equation}\label{8302}
\begin{split}
|J_{3}|^2=J_{3}\bar{J}_{3}=\iint{e^{i\theta_k(\phi(\xi_1+\tau_1,\xi_2,\xi_3)-\phi(\xi))}\Upsilon_3(\xi,\tau_1)}d\xi_{1}d\tau_{1},
\end{split}
\end{equation}
where
\begin{equation*}
\Upsilon_3(\xi,\tau_1)=\Xi(\xi_1+\tau_1,\xi_2,\xi_3)\bar{\Xi}(\xi).
\end{equation*}
We introduce the operator $L_{3}$ by
\begin{equation*}
L_{3} f(\xi)=\frac{1}{i\theta_k}g_{3}\cdot\partial_{\xi_1} f(\xi),~~\mathrm{where}~~g_{3}=\frac{1}{\partial_{\xi_1}\phi(\xi_1+\tau_1,\xi_2,\xi_3)-\partial_{\xi_1}\phi(\xi)}.
\end{equation*}
Then $L_{3}^{\top}f=-\frac{1}{i\theta_k}\partial_{\xi_1}(g_3 f)$ and
\begin{equation}\label{e9126}
\begin{split}
|J_{3}|^2&=\iint{e^{i\theta_k(\phi(\xi_1+\tau_1,\xi_2,\xi_3)-\phi(\xi))}(L^{\top}_3)^2\Upsilon_3(\xi,\tau_1)}d\xi_{1}d\tau_{1}\\
&\leq \iint{\frac{C}{|\theta_k|^2}\sum\limits_{j_1+j_2=0}^{2}|\partial_{\xi_1}^{j_1}g_{3}|\cdot|\partial_{\xi_1}^{j_2}g_{3}|}d\xi_{1}d\tau_{1}.
\end{split}
\end{equation}
By the differential mean value theorem we have
\begin{equation}\label{e9124}
|g_{3}|\leq\frac{C}{\inf_{\xi\in\Omega_{1}}|\partial_{\xi_1}^2q(\xi)|\cdot|\tau_{1}|},~~~|\partial_{\xi_1} g_{3}|\leq \frac{C\sup_{\xi\in\Omega_{1}}|\partial_{\xi_1}^3q(\xi)|}{\inf_{\xi\in\Omega_{1}}|\partial_{\xi_1}^2q(\xi)|^2\cdot|\tau_{1}|},
\end{equation}
and
\begin{equation}\label{e9125}
\begin{split}
|\partial_{\xi_1}^{2} g_{3}|&\leq\frac{C}{|\tau_{1}|}\bigg( \frac{\sup_{\xi\in\Omega_{1}}|\partial_{\xi_1}^4q(\xi)|}{\inf_{\xi\in\Omega_{1}}|\partial^2_{\xi_1}q(\xi)|^2}+\frac{\sup_{\xi\in\Omega_{1}}|\partial_{\xi_1}^3q(\xi)|^{2}}{\inf_{\xi\in\Omega_{1}}|\partial_{\xi_1}^2q(\xi)|^3}\bigg).
\end{split}
\end{equation}
Thus it suffices to estimate the upper bounds of $|\partial_{\xi_1}^3q(\xi)|$, $|\partial_{\xi_1}^4q(\xi)|$ and the lower bound of $|\partial_{\xi_1}^2q(\xi)|$ on $\Omega_{1}$. Noting that $\frac{1}{2}\sigma_{k}\leq A,B\leq 2\sigma_k$ and $|\xi_1|\leq 4$, we have
\begin{equation*}
\begin{split}
|\partial_{\xi_1}^2q|&=\Big|\Big(\frac{1}{A}+\frac{1}{B}\Big)\cdot\Big(1-\Big(\frac{1}{A^2}+\frac{1}{B^2}-\frac{1}{AB}\Big)\xi_1^2\Big)\Big|\geq C\Big(\frac{1}{A}+\frac{1}{B}\Big)\cdot\Big(1-\frac{12\times16}{\sigma_k^2}\Big)\geq \frac{C}{\sigma_k}.
\end{split}
\end{equation*}
Furthermore, we can check that
\begin{equation*}
|\partial_{\xi_1}^3q|,~|\partial_{\xi_1}^4q|\leq C\Big(\Big|\frac{1}{A^3}+\frac{1}{B^3}\Big|+\Big|\frac{1}{A^5}+\frac{1}{B^5}\Big|+\Big|\frac{1}{A^7}+\frac{1}{B^7}\Big|\Big)\leq \frac{C}{\sigma_k^3}.
\end{equation*}
Thus, introducing the above estimates into \eqref{e9124} and \eqref{e9125}, and plugging the results into \eqref{e9126} we get
\begin{equation}\label{8301}
\begin{split}
|J_{3}|^2\leq C\iint{\Big(\frac{\sigma_k}{\theta_k\tau_1}\Big)^2}d\xi_{1}d\tau_{1}\leq C\int_{-8}^8{\Big(\frac{\sigma_k}{\theta_k\tau_1}\Big)^2}d\tau_{1}.
\end{split}
\end{equation}
Noting also that
\begin{equation*}
|J_3|^2\leq C\int^8_{-8}{1}d\tau_1.
\end{equation*}
Thus
\begin{equation*}
\begin{split}
|J_{3}|\leq \bigg(\int_{-8}^8{\frac{C}{1+|\frac{\theta_k\tau_1}{\sigma_k}|^2}}d\tau_{1}\bigg)^{1/2}\leq C\Big(\frac{\sigma_k}{|\theta_k|}\Big)^{1/2}.
\end{split}
\end{equation*}
Plugging the above estimate into $I_{1,1}$ we can finally obtain that
\begin{equation}\label{e9113}
|I_{1,1}|\leq C2^{3k}\Big(\frac{\sigma_k}{|\theta_k|}\Big)^{1/2}
\leq C2^{3k}\Big(\frac{\sigma_k^2}{|\theta_k|}\Big)^{1/2}
\end{equation}
for $q=A+B$ and $\sigma_k>60$.

Combining \eqref{928} and \eqref{e9113} we get
\begin{equation}\label{929}
|I_{1,1}|\leq C2^{3k}\Big(\frac{\sigma_k^2}{|\theta_k|}\Big)^{1/2}\leq C2^{3k}\Big(\frac{\nu^2\delta}{2^{3k}|t|}\Big)^{1/2}
\end{equation}
for $q=A\pm B$ and $\sigma_k> 60$.

Thus, combining \eqref{9118}, \eqref{925} and \eqref{929} we have that
\begin{equation*}
|I_{1,1}|\leq C2^{3k}\begin{cases}
\big(\frac{\varepsilon}{|t|}\big)^{1/2},~\mathrm{if}~\sigma_k\leq 60;\\
\big(\frac{\nu^2\delta}{2^{3k}|t|}\big)^{1/2},~\mathrm{if}~\sigma_k> 60.
\end{cases}
\end{equation*}
Finally, noting that on $\mathrm{supp}\,\psi_1(\xi)(1-\psi_{1,1}(\xi_1))$ we have $|\xi_1|\geq\frac{1}{4}\sqrt{\frac{3}{2}}$, the estimate for $I_{1,2}$ can be derived similarly by switching the roles of $\xi_1$ and $\xi_2$ in the above proof since $\xi_1$ and $\xi_2$ are symmetric in $q(\xi)$. Thus we conclude with
\begin{equation}\label{9211}
|I_{1}|\leq C2^{3k}\begin{cases}
\big(\frac{\varepsilon}{|t|}\big)^{1/2},~\mathrm{if}~\sigma_k\leq 60;\\
\big(\frac{\nu^2\delta}{2^{3k}|t|}\big)^{1/2},~\mathrm{if}~\sigma_k> 60
\end{cases}
\end{equation}
for $q=A\pm B$.

\subsection{Estimates for $I_2$}\label{sec3.2}
In this subsection we construct the estimates for $I_2(x)$. Noting that
\begin{equation*}
I_2(x)=2^{3k}\int_{\mathbb{R}^{3}}{e^{i\theta_k\phi(\xi)}\psi_2(\xi)}d\xi,
\end{equation*}
where
\begin{equation*}
\phi(\xi)=\phi(\xi_1,\xi_2,\xi_3):=x\cdot\xi+q(\xi),~~\mathrm{and}~~\mathrm{supp}\,\psi_{2}\subset \Big\{\frac{1}{2}\leq|\xi|\leq 3,~|\xi_3|\geq\frac{1}{29}\Big\}.
\end{equation*}

\subsubsection{Estimates for high frequencies}
To construct the estimates for high frequencies, we break $I_{2}$ into three parts.  Let $\psi_{2,j}\in C_{0}^{\infty}(\mathbb{R}^3)$, $j=1,2$, such that
\begin{equation*}
\begin{split}
&0\leq \psi_{2,j}\leq 1,~|\nabla_\xi\psi_{2,j}|\leq C,\\
&\mathrm{supp}\,\psi_{2,1}\subset\Big\{|\xi|\leq 4,~|\xi_1|\leq \frac{1}{185}\Big\},~\psi_{2,1}=1~\mathrm{on}~\Big\{|\xi_1|\leq \frac{1}{195}\Big\}\cap\mathrm{supp}\,\psi_2,\\
&\mathrm{supp}\,\psi_{2,2}\subset\Big\{|\xi|\leq4,~|\xi_1|\geq\frac{1}{195},~|\xi_2|\leq \frac{1}{185}\Big\},\\
&~~~~~~~\psi_{2,2}=1~\mathrm{on}~\Big\{|\xi_1|\geq\frac{1}{185},~|\xi_2|\leq \frac{1}{195}\Big\}\cap\mathrm{supp}\,\psi_2,\\
&\mathrm{and}~\sum\nolimits_{j=1}^{2}\psi_{2,j}(\xi)=1~\mathrm{on}~\Big(\Big\{|\xi_1|\leq\frac{1}{195}\Big\} \cup \Big\{|\xi_2|\leq\frac{1}{195}\Big\}\Big)\cap\mathrm{supp}\,\psi_2.
\end{split}
\end{equation*}
Denoting by
\begin{equation*}
\Xi_{2,1}:=\psi_2\psi_{2,1},~~\Xi_{2,2}:=\psi_2\psi_{2,2},~~\mathrm{and}~~\Xi_{2,3}:=\psi_2\cdot(1-\psi_{2,1}-\psi_{2,2}).
\end{equation*}
Then
\begin{equation}\label{9110}
I_2(x)=\sum\nolimits_{j=1}^3I_{2,j}(x),~~\mathrm{where}~~I_{2,j}(x)=2^{3k}\int_{\mathbb{R}^{3}}{e^{i\theta_k\phi(\xi)}\Xi_{2,j}(\xi)}d\xi.
\end{equation}
We estimate $I_{2,1}$, $I_{2,2}$ and $I_{2,3}$, respectively. Noting that $\sigma_k\leq \frac{1}{60}$ here.

\textbf{Estimate for $I_{2,1}$}. We have
\begin{equation*}
\mathrm{supp}\,\Xi_{2,1}\subset \Big\{\frac{1}{2}\leq|\xi|\leq 3,~|\xi_1|\leq\frac{1}{185},~|\xi_3|\geq\frac{1}{29}\Big\}.
\end{equation*}
By using a partition of unity to cover the support of the above $\Xi_{2,1}$, we may assume that the support of $\Xi_{2,1}$ is sufficiently small such that for any $\xi,\eta\in\mathrm{supp}\,\Xi_{2,1}$, the line segment connecting $\xi$ and $\eta$ lies wholly in
\begin{equation*}
\Omega_{2,1}:=\Big\{\frac{1}{4}\leq|\xi|\leq 4,~|\xi_1|\leq\frac{1}{180},~|\xi_3|\geq\frac{1}{30}\Big\}.
\end{equation*}
The decay estimate of $I_{2,1}$ is derived from the oscillation integral on $\mathbb{R}_{\xi_1}$. Thus the proofs are similar to those for $I_{1,1}$ with $q=A+B$ and low frequencies. In spirit of the discussions in \eqref{8302}-\eqref{e9113}, we are led to estimate the upper bounds of $|\partial_{\xi_1}^3q(\xi)|$, $|\partial_{\xi_1}^4q(\xi)|$ and the lower bound of $|\partial_{\xi_1}^2q(\xi)|$ on $\Omega_{2,1}$.
\begin{itemize}
\item For $q=A-B$ We have
\begin{equation*}
\begin{split}
|\partial_{\xi_1}^2q|&=\Big|\Big(\frac{1}{A}-\frac{1}{B}\Big)\cdot\Big(1-\Big(\frac{1}{A^2}+\frac{1}{B^2}+\frac{1}{AB}\Big)\xi_1^2\Big)\Big|\geq\Big|\frac{1}{A}-\frac{1}{B}\Big|\Big(1-3\times60^2\times\frac{1}{180^2}\Big)\\
&\geq C\frac{|A^2-B^2|}{AB(A+B)}\geq C\sigma_k
\end{split}
\end{equation*}
and
\begin{equation*}
|\partial_{\xi_1}^3q|,~|\partial_{\xi_1}^4q|\leq C\Big|\frac{A^2-B^2}{AB(A+B)}\Big|\leq C\sigma_k
\end{equation*}
by observing that $A,B\geq \min\{|\xi_3-\sigma_k|,|\xi_3+\sigma_k|\}\geq \frac{1}{60}$.
\item For $q=A+B$ we have
\begin{equation*}
\begin{split}
|\partial_{\xi_1}^2q|&=\Big|\Big(\frac{1}{A}+\frac{1}{B}\Big)\cdot\Big(1-\Big(\frac{1}{A^2}+\frac{1}{B^2}-\frac{1}{AB}\Big)\xi_1^2\Big)\Big|\geq C\Big(\frac{1}{A}+\frac{1}{B}\Big)\geq C,
\end{split}
\end{equation*}
and
\begin{equation*}
|\partial_{\xi_1}^3q|,~|\partial_{\xi_1}^4q|\leq C.
\end{equation*}
\end{itemize}
Thus, in similar fashion to the discussions in \eqref{8302}-\eqref{e9113}, we can finally get
\begin{equation*}
|I_{2,1}|\leq C2^{3k}\begin{cases}
\Big(\frac{1}{|\theta_k|\sigma_k}\Big)^{1/2},~\mathrm{for}~q=A-B;\\
\Big(\frac{1}{|\theta_k|}\Big)^{1/2},~\mathrm{for}~q=A+B,
\end{cases}
\end{equation*}
which implies that
\begin{equation}\label{e9127}
|I_{2,1}|\leq C2^{3k}
\Big(\frac{1}{|\theta_k|\sigma_k}\Big)^{1/2}
\end{equation}
for $q=A\pm B$ and $\sigma_k\leq \frac{1}{60}$.

\textbf{Estimate for $I_{2,2}$}. Noting that on $\mathrm{supp}\,\Xi_{2,2}$ we have $|\xi_2|\leq\frac{1}{185}$ and $\xi_1$ and $\xi_2$ are symmetric in $q(\xi)$, the estimate for $I_{2,2}$ can be constructed similarly by switching the roles of $\xi_1$ and $\xi_2$ in the above discussions for $I_{2,1}$. Thus we also have
\begin{equation}\label{9112}
|I_{2,2}|\leq C2^{3k}\Big(\frac{1}{|\theta_k|\sigma_k}\Big)^{1/2}
\end{equation}
for $q=A\pm B$ and $\sigma_k\leq\frac{1}{60}$.

\textbf{Estimate for $I_{2,3}$}. We have
\begin{equation*}
\mathrm{supp}\,\Xi_{2,3}\subset \Big\{\frac{1}{2}\leq|\xi|\leq 3,~|\xi_1|\geq\frac{1}{195},~|\xi_2|\geq\frac{1}{195},~|\xi_3|\geq\frac{1}{29}\Big\}.
\end{equation*}
By using a partition of unity to cover the support of the above $\Xi_{2,3}$, we may assume that the support of $\Xi_{2,3}$ is sufficiently small such that for any $\xi,\eta\in\mathrm{supp}\,\Xi_{2,3}$, the line segment connecting $\xi$ and $\eta$ lies wholly in
\begin{equation*}
\Omega_{2,3}:=\Big\{\frac{1}{4}\leq|\xi|\leq 4,~|\xi_1|\geq\frac{1}{200},~|\xi_2|\geq\frac{1}{200},~|\xi_3|\geq\frac{1}{30}\Big\}.
\end{equation*}
The decay estimate for $I_{2,3}$ can be obtained by switching the roles of $\xi_1$ and $\xi_3$ in the estimate of $I_{1,1}$ with $q=A-B$ and high frequencies. Indeed, for $\xi\in\Omega_{2,3}$ and $\sigma_k\leq \frac{1}{60}$ we can verify that:
\begin{itemize}
\item for $q=A-B$,
\begin{equation*}
|\partial_{\xi_1\xi_2}q|\geq C\Big|\frac{1}{A}-\frac{1}{B}\Big|\Big(\frac{1}{A}+\frac{1}{B}+\frac{1}{AB}\Big)\geq C\Big|\frac{A^2-B^2}{AB(A+B)}\Big|\geq C\sigma_k,
\end{equation*}
and
\begin{equation*}
|\partial_{\xi_2}\partial_{\xi_1}^2q|,~|\partial_{\xi_2}\partial_{\xi_1}^3q|\leq C\Big|\frac{1}{A}-\frac{1}{B}\Big|\leq C\sigma_k,
\end{equation*}
\item for $q=A+B$,
\begin{equation*}
|\partial_{\xi_1\xi_2}q|\geq C\Big(\frac{1}{A^3}+\frac{1}{B^3}\Big)\geq C,
\end{equation*}
and
\begin{equation*}
|\partial_{\xi_2}\partial_{\xi_1}^2q|,~|\partial_{\xi_2}\partial_{\xi_1}^3q|\leq C\Big(\frac{1}{A}+\frac{1}{B}\Big)\leq C.
\end{equation*}
\end{itemize}
Thus, switching the roles of $\xi_1$ and $\xi_3$ in the discussions in \eqref{e9112}-\eqref{9116} we can finally obtain that
\begin{equation*}
|I_{2,3}|\leq C2^{3k}\begin{cases}
\Big(\frac{1}{|\theta_k|\sigma_k}\Big)^{1/2},~\mathrm{for}~q=A-B;\\
\Big(\frac{1}{|\theta_k|}\Big)^{1/2},~\mathrm{for}~q=A+B,
\end{cases}
\end{equation*}
which implies
\begin{equation}\label{9113}
|I_{2,3}|\leq C2^{3k}\Big(\frac{1}{|\theta_k|\sigma_k}\Big)^{1/2}
\end{equation}
for $q=A\pm B$ and $\sigma_k\leq\frac{1}{60}$.

Combining \eqref{9110}, \eqref{e9127}, \eqref{9112} and \eqref{9113} we conclude with
\begin{equation}\label{9114}
|I_{2}|\leq C2^{3k}\Big(\frac{1}{|\theta_k|\sigma_k}\Big)^{1/2}\leq C2^{3k}\Big(\frac{\varepsilon}{|t|}\Big)^{1/2}
\end{equation}
for $q=A\pm B$ and $\sigma_{k}\leq\frac{1}{60}$.

\subsubsection{Estimates for middle frequencies}
Now we construct the estimate for $I_2(x)$ with middle frequencies such that $\frac{1}{60}< \sigma_k\leq60$. Noting that $I_2(x)$ is rotation invariant in the plane $\mathbb{R}_{x_h}^2$, we may assume that $x_2=0$. We may also assume that $\theta_k>0$ in $I_2(x)$. Indeed, if $\theta_k<0$, then we can substitute $q$ by $-q$ in the following discussions, and the results below still hold true. We introduce the operator $L_{4}$ by
\begin{equation*}
L_{4}=\frac{1}{1+\theta_k\cdot|\partial_{\xi_2}q|^2}\Big(\mathrm{id}-i\partial_{\xi_2}q\cdot\partial_{\xi_2}\Big).
\end{equation*}
Then we can check that
\begin{equation}\label{7239}
I_2(x)=2^{3k}\int_{\mathbb{R}^{3}}{L_{4}(e^{i\theta_k(x_1\xi_1+x_3\xi_3+q(\xi))})\psi_2(\xi)}d\xi=2^{3k}\int_{\mathbb{R}^{3}}{e^{i\theta_k(x_1\xi_1+x_3\xi_3+q(\xi))}\cdot L_{4}^{\top}\psi_2(\xi)}d\xi,
\end{equation}
where
\begin{equation*}
L_{4}^{\top}\psi_2=\frac{1}{1+\theta_k\cdot|\partial_{\xi_2}q|^2}\Big(\psi_2+i\partial_{\xi_2}q\cdot\partial_{\xi_2}\psi_2+i\partial^2_{\xi_2}q\cdot\psi_2-\frac{2i\theta_{k}\cdot|\partial_{\xi_2}q|^2\cdot\partial_{\xi_2}^2q}{1+\theta_k\cdot|\partial_{\xi_2}q|^2}\psi_2\Big).
\end{equation*}
\begin{itemize}
\item For $q=A-B$ we have, on $\mathrm{supp}\,\psi_2$,
\begin{equation*}
\begin{split}
&|\partial_{\xi_2}q|=\frac{|A^2-B^2|}{AB(A+B)}|\xi_2|\geq C\sigma_k|\xi_2|\geq C|\xi_2|.
\end{split}
\end{equation*}
\item For $q=A+B$ we have
\begin{equation*}
|\partial_{\xi_2}q|=\Big(\frac{1}{A}+\frac{1}{B}\Big)|\xi_2|\geq C|\xi_2|.
\end{equation*}
\end{itemize}
Furthermore, for both $q=A-B$ and $q=A+B$ we can check that
\begin{equation*}
\begin{split}
&|\partial_{\xi_2}q|\leq \frac{|\xi_2|}{A}+\frac{|\xi_2|}{B}\leq C,~~\mathrm{and}~~|\partial_{\xi_2}^2q|\leq C\bigg(\frac{1}{A}+\frac{1}{B}\bigg).
\end{split}
\end{equation*}
Thus, after introducing the above estimates into \eqref{7239} we get
\begin{equation}\label{72310}
|I_2|\leq C\cdot2^{3k}\iint_{\mathrm{supp}\,\psi_2}{\frac{1}{1+C\theta_k|\xi_2|^2}\Big(1+\frac{1}{A}+\frac{1}{B}\Big)}d\xi.
\end{equation}
We can directly check that
\begin{equation}\label{72311}
\iint_{\mathrm{supp}\,\psi_2}{\frac{1}{1+C\theta_k|\xi_2|^2}}d\xi\leq C\int_{\mathbb{R}}{\frac{1}{1+C\theta_k|\xi_2|^2}}d\xi_2\leq \frac{C}{\sqrt{\theta_k}},
\end{equation}
and
\begin{equation*}
\begin{split}
&\iint_{\mathrm{supp}\,\psi_2}{\frac{1}{1+C\theta_k|\xi_2|^2}\cdot\frac{1}{A}}d\xi\leq\int_{\mathbb{R}}{\frac{1}{1+C\theta_k|\xi_2|^2}}d\xi_2\cdot\int_{-3}^{3}\int_{-3}^{3}{\frac{1}{\sqrt{\xi_1^2+(\xi_3+\sigma_k)^2}}}d\xi_1d\xi_3.
\end{split}
\end{equation*}
We define $\Omega_{2,4}:=(-3,3)\times(-3+\sigma_k,3+\sigma_k)$ and $\tilde{B}(0,1):=\{x\in\mathbb{R}^2:|x|\leq1\}$. Then
\begin{equation*}
\begin{split}
&\int_{-3}^{3}\int_{-3}^{3}{\frac{1}{\sqrt{\xi_1^2+(\xi_3+\sigma_k)^2}}}d\xi_1d\xi_3=\iint_{\Omega_{2,4}}{\frac{1}{\sqrt{\xi_1^2+y^2}}}d\xi_1dy\\
&~~~\leq \iint_{\Omega_{2,4}-\tilde{B}(0,1)}{1}d\xi_1dy+\iint_{\tilde{B}(0,1)}{\frac{1}{\sqrt{\xi_1^2+y^2}}}d\xi_1dy\leq C.
\end{split}
\end{equation*}
Thus
\begin{equation}\label{72312}
\iint_{\mathrm{supp}\,\psi_2}{\frac{1}{1+C\theta_k|\xi_2|^2}\cdot\frac{1}{A}}d\xi\leq \frac{C}{\sqrt{\theta_k}},
\end{equation}
and similarly,
\begin{equation}\label{72313}
\begin{split}
&\iint_{\mathrm{supp}\,\psi_2}{\frac{1}{1+C\theta_k|\xi_2|^2}\cdot\frac{1}{B}}d\xi\leq \frac{C}{\sqrt{\theta_k}}.
\end{split}
\end{equation}
Plugging \eqref{72311}, \eqref{72312} and \eqref{72313} into \eqref{72310} yields
\begin{equation}\label{7214}
|I_2|\leq
\frac{C\cdot 2^{3k}}{\sqrt{\theta_k}}\leq C2^{3k}\cdot
\Big(\frac{\delta}{2^k|t|}\Big)^{1/2}=C2^{3k}\cdot
\Big(\frac{\sigma_k\varepsilon}{|t|}\Big)^{1/2}\leq C2^{3k}\cdot
\Big(\frac{\varepsilon}{|t|}\Big)^{1/2}
\end{equation}
for $q=A\pm B$ and $\frac{1}{60}<\sigma_k\leq 60$.

\subsubsection{Estimates for low frequencies}
In this case the decay estimate is derived from the oscillation integral on $\mathbb{R}_{\xi_1}$. Thus, similar to the discussions in \eqref{8302}-\eqref{e9113}, it suffices to estimate the upper bounds of $|\partial_{\xi_1}^3q(\xi)|$, $|\partial_{\xi_1}^4q(\xi)|$ and the lower bound of $|\partial_{\xi_1}^2q(\xi)|$ on some area which includes and is slightly bigger than $\mathrm{supp}\,\psi_2$, say:
\begin{equation*}
\Omega_{2,5}:=\Big\{\frac{1}{4}\leq|\xi|\leq 4,~|\xi_3|\geq\frac{1}{30}\Big\}.
\end{equation*}
Noting that $\sigma_k>60$ here.
\begin{itemize}
\item For $q=A-B$ we have
\begin{equation*}
\begin{split}
|\partial_{\xi_1}^2q|&=\Big|\Big(\frac{1}{A}-\frac{1}{B}\Big)\cdot\Big(1-\Big(\frac{1}{A^2}+\frac{1}{B^2}+\frac{1}{AB}\Big)\xi_1^2\Big)\Big|\geq\Big|\frac{1}{A}-\frac{1}{B}\Big|\cdot\Big(1-\frac{12\times16}{\sigma_k^2}\Big)\\
&\geq C\frac{|\xi_3|\sigma_k}{AB(A+B)}\geq \frac{C}{\sigma_k^2}
\end{split}
\end{equation*}
by observing that $\frac{1}{2}\sigma_k\leq A,B\leq 2\sigma_k$ and $|\xi_1|\leq 4$.
\item For $q=A+B$ we have
\begin{equation*}
\begin{split}
|\partial_{\xi_1}^2q|&=\Big|\Big(\frac{1}{A}+\frac{1}{B}\Big)\cdot\Big(1-\Big(\frac{1}{A^2}+\frac{1}{A^2}-\frac{1}{AB}\Big)\xi_1^2\Big)\Big|\geq C\Big(\frac{1}{A}+\frac{1}{B}\Big)\geq \frac{C}{\sigma_k}\geq \frac{C}{\sigma_k^2}.
\end{split}
\end{equation*}
\end{itemize}
For both $q=A-B$ and $q=A+B$ we have
\begin{equation*}
|\partial_{\xi_1}^3q|,~|\partial_{\xi_3}^4q|\leq C\Big(\Big|\frac{1}{A^3}+\frac{1}{B^3}\Big|+\Big|\frac{1}{A^5}+\frac{1}{B^5}\Big|+\Big|\frac{1}{A^7}+\frac{1}{B^7}\Big|\Big)\leq \frac{C}{\sigma_k^3}.
\end{equation*}
Thus, in similar fashion to the discussions in \eqref{8302}-\eqref{e9113}, we can finally get
\begin{equation}\label{9210}
|I_{2}|\leq C2^{3k}\Big(\frac{\sigma_k^2}{|\theta_k|}\Big)^{1/2}\leq C2^{3k}\Big(\frac{\nu^2\delta}{2^{3k}|t|}\Big)^{1/2}
\end{equation}
for $q=A\pm B$ and $\sigma_k> 60$.

Combining \eqref{9114}, \eqref{7214} and \eqref{9210} we conclude with
\begin{equation}\label{9212}
|I_{2}|\leq C2^{3k}\begin{cases}
\big(\frac{\varepsilon}{|t|}\big)^{1/2},~\mathrm{if}~\sigma_k\leq 60;\\
\big(\frac{\nu^2\delta}{2^{3k}|t|}\big)^{1/2},~\mathrm{if}~\sigma_k> 60.
\end{cases}
\end{equation}
Combining \eqref{9211} and \eqref{9212} we further have
\begin{equation}\label{9215}
|I|\leq C2^{3k}\begin{cases}
\big(\frac{\varepsilon}{|t|}\big)^{1/2},~\mathrm{if}~\sigma_k\leq 60;\\
\big(\frac{\nu^2\delta}{2^{3k}|t|}\big)^{1/2},~\mathrm{if}~\sigma_k> 60
\end{cases}\leq C2^{3k}\Big(\frac{\varepsilon\mathcal{M}_k}{|t|}\Big)^{1/2},
\end{equation}
where $\mathcal{M}_k$ is defined in \eqref{72317}. On the other hand, noting that $|I(x)|\leq C2^{3k}$, thus, combining this estimate with \eqref{9215} we finally obtain with \eqref{7231}. Now we have completed the proof of Proposition \ref{pro1}.

\section{A result of TT* argument}\label{sec3}
In this section we prove a proposition which has been used in Proposition \ref{p2}. Besides this, the proposition can also be applied to derive the space-time decay estimates for propagators in many other works such as \cite{Mu,Mu-S,B4265}.

For $(t,x)\in \mathbb{R}\times\mathbb{R}^d$ and some functions $\lambda(\xi)$ and $\psi(\xi)$, we define the operators
\begin{equation}\label{8620}
V(t)f(x)=\frac{1}{(2\pi)^{d}}\int_{\mathbb{R}^{d}}{e^{ix\cdot\xi+it\lambda(\xi)}\psi^{2}(\xi)\hat{f}(\xi)}d\xi,
\end{equation}
and
\begin{equation*}
\tilde{V}(t)f(x)=\frac{1}{(2\pi)^{d}}\int_{\mathbb{R}^{d}}{e^{ix\cdot\xi+it\lambda(\xi)}\psi(\xi)\hat{f}(\xi)}d\xi.
\end{equation*}
We assume that
\begin{equation}\label{861}
\|(V(t)f,\tilde{V}(t)f)\|_{L^{\infty}(\mathbb{R}^d)}\leq \frac{C\mu_1}{(1+\mu_{2}|t|)^{\sigma}}\|f\|_{L^{1}(\mathbb{R}^d)},
\end{equation}
and
\begin{equation}\label{851}
\|(V(t)f,\tilde{V}(t)f)\|_{L^{2}(\mathbb{R}^d)}\leq C\|f\|_{L^{2}(\mathbb{R}^d)},
\end{equation}
where $C$, $\mu_1$ and $\mu_2$ are independent of $t$, and $\sigma>0$. In the following, we use the notations
\begin{equation*}
\|f\|_{L^{q}_{t}}:=\|f\|_{L^{q}(\mathbb{R})},~~\|f\|_{L^{r}_{x}}:=\|f\|_{L^{r}(\mathbb{R}^d)},~~\mathrm{and}~~\|f\|_{L^{q}_{t}L^{r}_{x}}:=\|f\|_{L^{q}(\mathbb{R};L^{r}(\mathbb{R}^{d}))}.
\end{equation*}
We use $f\lesssim g$ to denote that $f\leq C g$ for some constant $C$ independent of $\mu_1$, $\mu_{2}$ and $\psi$. We begin from introducing some definitions \cite{Tao}:
\begin{definition}
We say that the exponent pair $(q,r)\in[2,\infty]^2$ is $\sigma$-admissible if $(q,r,\sigma)\neq(2,\infty,1)$ and
\begin{equation}\label{841}
\frac{2}{q}\leq\sigma\Big(1-\frac{2}{r}\Big).
\end{equation}
If equality holds in \eqref{841} we say that $(q,r)$ is sharp $\sigma$-admissible, otherwise we say that $(q,r)$ is non-sharp $\sigma$-admissible. Note in particular that when $\sigma>1$ the endpoint
\begin{equation*}
P_{\sigma}=\Big(2,\frac{2\sigma}{\sigma-1}\Big)
\end{equation*}
is sharp $\sigma$-admissible.
\end{definition}

\begin{proposition}\label{a1}
Let $V(t)$ be defined in \eqref{8620} and satisfy the estimates \eqref{861} and \eqref{851}. Then:\\
(i). for any $\sigma$-admissible pair $(q,r)$ such that $(q,r)\neq P_{\sigma}$ we have
\begin{equation}\label{862}
\|V(t)f\|_{L^q_t L^r_x}\lesssim C^{\frac{1}{2}}_1  \mu_1^{\frac{1}{2}-\frac{1}{r}}\mu_{2}^{-\frac{1}{q}}\|f\|_{L^2_x},~~\mathrm{where}~~C_1=\|\mathcal{F}^{-1}(\psi^{2})\|_{L^{1}_x};
\end{equation}
(ii). for any $\sigma$-admissible pair $(q,r)$ and any sharp $\sigma$-admissible pair $(\tilde{q},\tilde{r})$ such that $(q,r)\neq P_{\sigma}$ and $(\tilde{q},\tilde{r})\neq P_{\sigma}$ we have
\begin{equation}\label{865}
\bigg\|\int^{t}_{-\infty}{ V(t-s) F(s)}ds\bigg\|_{L^{q}_{t}L^{r}_{x}}\lesssim C_2^{\frac{1}{2}}(1+C_{2}^{\frac{1}{2}})^{\theta}C_{4,r}^{1-\theta}(\mu_1^{1/\sigma}\mu_{2}^{-1})^{\frac{1}{q}+\frac{1}{\tilde{q}}}\| F\|_{L^{\tilde{q}'}_{t}L^{\tilde{r}'}_{x}},
\end{equation}
where
\begin{equation}\label{8611}
\begin{split}
&C_2=\|\mathcal{F}^{-1}(\psi)\|_{L^{1}_x},~~C_{4,r}=\|\mathcal{F}^{-1}(\psi)\|_{L^{(\frac{1}{2}+\frac{1}{r})^{-1}}_x},~~\mathrm{and}~~\theta=\frac{2r}{q\sigma(r-2)};
\end{split}
\end{equation}
(iii). for any $\sigma$-admissible pairs $(q,r)$ and $(\tilde{q},\tilde{r})$ such that $(q,r)\neq P_{\sigma}$ and $(\tilde{q},\tilde{r})\neq P_{\sigma}$ we have
\begin{equation}\label{8123}
\bigg\|\int^{t}_{-\infty}{ V(t-s) F(s)}ds\bigg\|_{L^{q}_{t}L^{r}_{x}}\lesssim (C_2^{\frac{1}{2}}C_{4,r})^{1-\theta}C^{\theta}_{5,r,\tilde{r}}\mu_1^{(\frac{1}{2}-\frac{1}{\tilde{r}})(1-\theta)+\frac{1}{\sigma}(\frac{1}{q_0}+\frac{1}{\tilde{q}})\theta}\mu_{2}^{-(\frac{1}{\tilde{q}}+\frac{1}{q})}\| F\|_{L^{\tilde{q}'}_{t}L^{\tilde{r}'}_{x}},
\end{equation}
where
\begin{equation}\label{8615}
C_{5,r,\tilde{r}}=C_2^{\frac{1}{2}}\times(1+C_{2}^{\frac{1}{2}})\times\big(1+C_{4,r}+C_{4,\tilde{r}}\big),~~q_0=\frac{2r}{\sigma(r-2)},
\end{equation}
and $C_2$, $C_{4,\bigcdot}$, $\theta$ are defined in \eqref{8611}.

\end{proposition}

\begin{proof}
The proofs are based on the standard TT* argument. Here we pay our attention to the effects of $\mu_1$, $\mu_2$ and $\psi$ on the final space-time estimates of $V(t)$. First, by \eqref{861}, \eqref{851} and Riesz's interpolation theorem we obtain for $2\leq r\leq \infty$ that
\begin{equation*}
\|V(t) f\|_{L^{r}_x}\lesssim \bigg|\frac{\mu_1}{(1+\mu_{2}|t|)^{\sigma}}\bigg|^{1-\frac{2}{r}}\| f\|_{L^{r'}_x}.
\end{equation*}
Thus
\begin{equation}\label{4191}
\begin{split}
|\langle V(-s) F(s), V(-t) G(t)\rangle_{L^{2}}|&=|\langle V(t-s) F(s),\mathcal{F}^{-1}(\psi^{2})\ast G(t)\rangle_{L^{2}}|\\
&\lesssim \| V(t-s) F(s)\|_{L^{r}_x}\|\mathcal{F}^{-1}(\psi^{2})\|_{L^{1}_x}\| G(t)\|_{L^{r'}_x}\\
&\lesssim C_1\bigg|\frac{\mu_{1}}{(1+\mu_{2}|t-s|)^{\sigma}}\bigg|^{1-\frac{2}{r}}\| F(s)\|_{L^{r'}_x}\| G(t)\|_{L^{r'}_x},
\end{split}
\end{equation}
where $C_1$ is defined in \eqref{862}. For  the sharp $\sigma$-admissible pair $(q,r)=(\infty,2)$ we get from \eqref{4191} that
\begin{equation*}
\int_{\mathbb{R}}\int_{\mathbb{R}}{|\langle V(-s) F(s), V(-t) G(t)\rangle_{L^{2}}|}dsdt\lesssim C_1\| F\|_{L^{1}_{t}L^{2}_{x}}\| G\|_{L^{1}_{t}L^{2}_{x}}.
\end{equation*}
In the case $\frac{2}{q}=\sigma(1-\frac{2}{r})$ with $(q,r)\neq(\infty,2)$ and $(q,r)\neq P_{\sigma}$ we have by the Hardy-Littlewood-Sobolev inequality \cite{BB50},
\begin{equation*}
\begin{split}
&\int_{\mathbb{R}}\int_{\mathbb{R}}{|\langle V(-s) F(s), V(-t) G(t)\rangle_{L^{2}}|}dsdt\\
\lesssim& C_1\big(\mu_{1} \mu_{2}^{-\sigma}\big)^{1-\frac{2}{r}}\int_{\mathbb{R}}\int_{\mathbb{R}}\frac{1}{|t-s|^{\frac{2}{q}}}\| F(s)\|_{L^{r'}_x}\| G(t)\|_{L^{r'}_x}dsdt\\
\lesssim&C_1\big(\mu_{1} \mu_{2}^{-\sigma}\big)^{1-\frac{2}{r}}\int_{\mathbb{R}}\frac{1}{|t-s|^{\frac{2}{q}}}\| F(s)\|_{L^{r'}_x}ds\bigg\|_{L^{q}_{t}}\| G\|_{L^{q'}_{t}L_{x}^{r'}}\\
\lesssim&C_1\big(\mu_{1} \mu_{2}^{-\sigma}\big)^{1-\frac{2}{r}}\| F\|_{L^{q'}_{t}L^{r'}_{x}}\| G\|_{L^{q'}_{t}L^{r'}_{x}}.
\end{split}
\end{equation*}
When $\frac{2}{q}<\sigma(1-\frac{2}{r})$, it follows from \eqref{4191} that
\begin{equation*}
\begin{split}
&\int_{\mathbb{R}}\int_{\mathbb{R}}{|\langle V(-s) F(s), V(-t) G(t)\rangle_{L^{2}}|}dsdt\\
\lesssim &C_1  \mu_{1}^{1-\frac{2}{r}}\int_{\mathbb{R}}\int_{\mathbb{R}}\frac{1}{(1+\mu_{2}|t-s|)^{\sigma(1-\frac{2}{r})}}\| F(s)\|_{L^{r'}_x}\| G(t)\|_{L^{r'}_x}dsdt\\
\lesssim&C_1 \mu_{1}^{1-\frac{2}{r}}\bigg\|\int_{\mathbb{R}}\frac{1}{(1+\mu_{2} |t-s|)^{\sigma(1-\frac{2}{r})}}\| F(s)\|_{L^{r'}_x}ds\bigg\|_{L^{q}_{t}}\| G\|_{L^{q'}_{t}L_{x}^{r'}}\\
\lesssim&C_1 \mu_{1}^{1-\frac{2}{r}}\bigg\|\frac{1}{(1+\mu_{2} |t|)^{\sigma(1-\frac{2}{r})}}\bigg\|_{L^{\frac{q}{2}}_{t}}\| F\|_{L^{q'}_{t}L^{r'}_{x}}\| G\|_{L^{q'}_{t}L_{x}^{r'}}\\
\lesssim&C_1 \mu_{1}^{1-\frac{2}{r}}\mu_{2}^{-\frac{2}{q}}\| F\|_{L^{q'}_{t}L^{r'}_{x}}\| G\|_{L^{q'}_{t}L_{x}^{r'}}.\\
\end{split}
\end{equation*}
Now we have showed that for any $\sigma$-admissible pair $(q,r)$ such that $(q,r)\neq P_{\sigma}$, it holds
\begin{equation}\label{863}
\int_{\mathbb{R}}\int_{\mathbb{R}}{|\langle V(-s) F(s), V(-t) G(t)\rangle_{L^{2}}|}dsdt\lesssim C_1\mu_{1}^{1-\frac{2}{r}}\mu_{2}^{-\frac{2}{q}}\| F\|_{L^{q'}_{t}L^{r'}_{x}}\| G\|_{L^{q'}_{t}L_{x}^{r'}},
\end{equation}
thus
\begin{equation}\label{4194}
\begin{split}
\bigg\|\int_{\mathbb{R}}{ V(-s) F(s)}ds\bigg\|_{L^{2}_x}\lesssim C^{\frac{1}{2}}_1\mu_{1}^{\frac{1}{2}-\frac{1}{r}}\mu_{2}^{-\frac{1}{q}}\| F\|_{L^{q'}_{t}L^{r'}_{x}}.
\end{split}
\end{equation}
By the duality argument we can justify \eqref{862}. Indeed, for any $h\in L^{r'}(\mathbb{R}^{d})$ and $g\in L^{q'}(\mathbb{R})$, we have
\begin{equation}\label{11161}
\begin{split}
&\int_{\mathbb{R}}\int_{\mathbb{R}^{d}}{ V(t)f(x)\cdot\overline{h(x)g(t)}}dxdt=\int_{\mathbb{R}}\int_{\mathbb{R}^{d}}f(x)\cdot\overline{ V(-t)h(x)}dx\cdot\overline{g(t)}dt\\
&=\overline{\int_{\mathbb{R}^{d}}\int_{\mathbb{R}} V(-t)h(x)g(t)dt\cdot\overline{f(x)}dx}\lesssim \bigg\|\int_{\mathbb{R}}{ V(-t)h(x)g(t)}dt\bigg\|_{L^{2}_{x}}\cdot\|f\|_{L^{2}_{x}}\\
&\lesssim C^{\frac{1}{2}}_1  \mu_{1}^{\frac{1}{2}-\frac{1}{r}}\mu_{2}^{-\frac{1}{q}}\|h\|_{L^{r'}_{x}}\|g\|_{L^{q'}_{t}}\|f\|_{L^{2}_{x}},
\end{split}
\end{equation}
which implies \eqref{862}.

Then, similar to \eqref{863} and \eqref{4194}, we can also verify that
\begin{equation}\label{473}
\begin{split}
&\bigg\|\int_{-\infty}^{t}{ V(-s) F(s)}ds\bigg\|_{L^{2}_{x}}\lesssim C^{\frac{1}{2}}_1  \mu_{1}^{\frac{1}{2}-\frac{1}{r}}\mu_{2}^{-\frac{1}{q}}\| F\|_{L^{q'}_{t}L^{r'}_{x}},\\
&\bigg\|\int_{-\infty}^{t}{\tilde{V}(-s) F(s)}ds\bigg\|_{L^{2}_{x}}\lesssim C^{\frac{1}{2}}_2  \mu_{1}^{\frac{1}{2}-\frac{1}{r}}\mu_{2}^{-\frac{1}{q}}\| F\|_{L^{q'}_{t}L^{r'}_{x}},
\end{split}
\end{equation}
\begin{equation*}
\int_{\mathbb{R}}\int_{\mathbb{R}}{|\langle \tilde{V}(-s) F(s), \tilde{V}(-t) G(t)\rangle_{L^{2}}|}dsdt\lesssim C_2 \mu_{1}^{1-\frac{2}{r}}\mu_{2}^{-\frac{2}{q}}\| F\|_{L^{q'}_{t}L^{r'}_{x}}\| G\|_{L^{q'}_{t}L_{x}^{r'}},
\end{equation*}
and
\begin{equation*}
\int_{\mathbb{R}}\int_{-\infty}^{t}{|\langle \tilde{V}(-s) F(s), \tilde{V}(-t) G(t)\rangle_{L^{2}}|}dsdt\lesssim C_2 \mu_{1}^{1-\frac{2}{r}}\mu_{2}^{-\frac{2}{q}}\| F\|_{L^{q'}_{t}L^{r'}_{x}}\| G\|_{L^{q'}_{t}L_{x}^{r'}},
\end{equation*}
where $C_2$ is defined in \eqref{8611}. Thus
\begin{equation}\label{475}
\begin{split}
\int_{\mathbb{R}}\int_{\mathbb{R}^{d}}{\bigg(\int_{-\infty}^{t}{ V(t-s) F(s)}ds\bigg)\cdot\overline{ G(t)}}dxdt&=\int_{\mathbb{R}}\int_{-\infty}^{t}{\big\langle\tilde{V}(-s) F(s),\tilde{V}(-t) G(t)\big\rangle_{L^{2}}}dsdt\\
&\lesssim C_2 \mu_{1}^{1-\frac{2}{r}}\mu_{2}^{-\frac{2}{q}}\| F\|_{L^{q'}_{t}L^{r'}_{x}}\| G\|_{L^{q'}_{t}L_{x}^{r'}}.
\end{split}
\end{equation}
As a result, for any $\sigma$-admissible pair $(\tilde{q},\tilde{r})$ such that $(\tilde{q},\tilde{r})\neq P_{\sigma}$ we obtain with
\begin{equation}\label{471}
\|\int_{-\infty}^{t}{ V(t-s) F(s)}ds\|_{L^{\tilde{q}}_{t}L^{\tilde{r}}_{x}}\lesssim C_2 \mu_{1}^{1-\frac{2}{\tilde{r}}}\mu_{2}^{-\frac{2}{\tilde{q}}}\| F\|_{L^{\tilde{q}'}_{t}L^{\tilde{r}'}_{x}}.
\end{equation}
On the other hand, from \eqref{473} and \eqref{851} we have
\begin{equation}\label{474}
\begin{split}
\int_{\mathbb{R}}\int_{-\infty}^{t}{\langle\tilde{V}(-s) F(s),\tilde{V}(-t) G(t)\rangle_{L^{2}}}dsdt\lesssim&\int_{\mathbb{R}}\bigg\|\int_{-\infty}^{t}{\tilde{V}(-s) F(s)}ds\bigg\|_{L^{2}_{x}}\|\tilde{V}(-t) G(t)\|_{L^{2}_{x}}dt\\
\lesssim& C^{\frac{1}{2}}_2 \mu_{1}^{\frac{1}{2}-\frac{1}{\tilde{r}}}\mu_{2}^{-\frac{1}{\tilde{q}}}\| F\|_{L^{\tilde{q}'}_{t}L^{\tilde{r}'}_{x}}\| G\|_{L^{1}_{t}L^{2}_{x}}.
\end{split}
\end{equation}
Thus, similar to \eqref{471} we get
\begin{equation}\label{472}
\bigg\|\int^{t}_{-\infty}{ V(t-s) F(s)}ds\bigg\|_{L^{\infty}_{t}L^{2}_{x}}\lesssim C^{\frac{1}{2}}_2 \mu_{1}^{\frac{1}{2}-\frac{1}{\tilde{r}}}\mu_{2}^{-\frac{1}{\tilde{q}}}\| F\|_{L^{\tilde{q}'}_{t}L^{\tilde{r}'}_{x}}.
\end{equation}
We remark here that, \eqref{474} and \eqref{472} hold for any $\sigma$-admissible pair $(\tilde{q},\tilde{r})$ such that $(\tilde{q},\tilde{r})\neq P_{\sigma}$. Noting that each sharp $\sigma$-admissible pair $(q,r)$ with $(q,r)\neq P_{\sigma}$ is an interpolation between $(\infty,2)$ and $(\tilde{q},\tilde{r})$ where $(\tilde{q},\tilde{r})$ is also a sharp $\sigma$-admissible pair such that $(\tilde{q},\tilde{r})\neq P_{\sigma}$. Thus, combining \eqref{471}, \eqref{472} and Riesz's interpolation theorem, we have for all sharp $\sigma$-admissible pairs $(q,r)$ and $(\tilde{q},\tilde{r})$ such that $(\tilde{q},\tilde{r})\neq P_{\sigma}$, $(\tilde{q},\tilde{r})\neq P_{\sigma}$ and $r\leq \tilde{r}$, that
\begin{equation}\label{476}
\begin{split}
\bigg\|\int^{t}_{-\infty}{ V(t-s) F(s)}ds\bigg\|_{L^{q}_{t}L^{r}_{x}}&\lesssim C_2^{\frac{1}{2}(1+\frac{\tilde{q}}{q})}(\mu_{1}^{1/\sigma}\mu_{2}^{-1})^{\frac{1}{\tilde{q}}+\frac{1}{q}}\| F\|_{L^{\tilde{q}'}_{t}L^{\tilde{r}'}_{x}}\\
&\lesssim C_{3}(\mu_{1}^{1/\sigma}\mu_{2}^{-1})^{\frac{1}{\tilde{q}}+\frac{1}{q}}\| F\|_{L^{\tilde{q}'}_{t}L^{\tilde{r}'}_{x}}.
\end{split}
\end{equation}
where
\begin{equation*}
C_{3}=C_2^{\frac{1}{2}}(1+C_2^{\frac{1}{2}}),
\end{equation*}
and we have used the fact that $C_2^{\tilde{q}/2q}\leq 1+C_2^{1/2}$ as $\tilde{q}\leq q$ in \eqref{476}. By the duality argument similar to \eqref{11161}, the condition $ r\leq\tilde{r}$ above can be removed. Hence the second inequality in \eqref{476} holds for all sharp $\sigma$-admissible pairs $(q,r)$ and $(\tilde{q},\tilde{r})$ such that $(\tilde{q},\tilde{r})\neq P_{\sigma}$ and $(\tilde{q},\tilde{r})\neq P_{\sigma}$.

Next, it is easy to see
\begin{equation*}
\|\tilde{V}(-t) G(t)\|_{L^{2}_x}\lesssim\|\psi\widehat{G(t)}\|_{L^{2}_x}\lesssim C_4\| G(t)\|_{L^{r'}_x},
\end{equation*}
where $C_{4,r}$ is defined in \eqref{8611}. Thus from \eqref{473} and \eqref{474} we obtain
\begin{equation*}
\int_{\mathbb{R}}\int_{-\infty}^{t}{\langle\tilde{V}(-s) F(s),\tilde{V}(-t) G(t)\rangle_{L^{2}}}dsdt\lesssim C^{\frac{1}{2}}_2 \mu_{1}^{\frac{1}{2}-\frac{1}{\tilde{r}}}\mu_{2}^{-\frac{1}{\tilde{q}}}C_{4,r}\| F\|_{L^{\tilde{q}'}_{t}L^{\tilde{r}'}_{x}}\| G\|_{L^{1}_{t}L^{r'}_{x}}.
\end{equation*}
Hence similar to \eqref{471} we get for any $\sigma$-admissible pair $(\tilde{q},\tilde{r})$ such that $(\tilde{q},\tilde{r})\neq P_{\sigma}$ and any $r\in[2,\infty]$ that
\begin{equation}\label{477}
\bigg\|\int^{t}_{-\infty}{ V(t-s) F(s)}ds\bigg\|_{L^{\infty}_{t}L^{r}_{x}}\lesssim C^{\frac{1}{2}}_2  \mu_{1}^{\frac{1}{2}-\frac{1}{\tilde{r}}}\mu_{2}^{-\frac{1}{\tilde{q}}}C_{4,r}\| F\|_{L^{\tilde{q}'}_{t}L^{\tilde{r}'}_{x}}.
\end{equation}
Since each $\sigma$-admission pair $(q,r)$ satisfying $(q,r)\neq P_{\sigma}$ is an interpolation between $(\infty,r)$ and any sharp $\sigma$-admission pair $(q_0,r)$ with $(q_0,r)\neq P_{\sigma}$, it follows from \eqref{477}, \eqref{476} with $q=q_0=\frac{2r}{\sigma(r-2)}$, and Riesz's interpolation theorem that the inequality
\begin{equation}\label{866}
\bigg\|\int^{t}_{-\infty}{ V(t-s) F(s)}ds\bigg\|_{L^{q}_{t}L^{r}_{x}}\lesssim (C_2^{\frac{1}{2}}C_{4,r})^{1-\theta}C_3^{\theta}(\mu_{1}^{1/\sigma}\mu_{2}^{-1})^{\frac{1}{q}+\frac{1}{\tilde{q}}}\| F\|_{L^{\tilde{q}'}_{t}L^{\tilde{r}'}_{x}}
\end{equation}
holds for any $\sigma$-admissible pair $(q,r)$ and any sharp $\sigma$-admissible pair $(\tilde{q},\tilde{r})$ such that $(q,r)\neq P_{\sigma}$ and $(\tilde{q},\tilde{r})\neq P_{\sigma}$, where $\theta$ is defined in \eqref{8611}. Thus \eqref{865} is proved.

Furthermore, as it was done in \eqref{476}, we amplify the constant in the right side of \eqref{866} to obtain that
\begin{equation}\label{868}
\bigg\|\int^{t}_{-\infty}{ V(t-s) F(s)}ds\bigg\|_{L^{q}_{t}L^{r}_{x}}\lesssim C_{5,r,\tilde{r}}(\mu_{1}^{1/\sigma}\mu_{2}^{-1})^{\frac{1}{q}+\frac{1}{\tilde{q}}}\| F\|_{L^{\tilde{q}'}_{t}L^{\tilde{r}'}_{x}}
\end{equation}
by observing
\begin{equation*}
\begin{split}
&(C_2^{\frac{1}{2}}C_{4,r})^{1-\theta}C_3^{\theta}=C_2^{\frac{1}{2}}(1+C_{2}^{\frac{1}{2}})^{\theta}C_{4,r}^{1-\theta}\leq C_{5,r,\tilde{r}},
\end{split}
\end{equation*}
where $C_{5,r,\tilde{r}}$ is defined in \eqref{8615}. Also, by the duality argument similar to \eqref{11161} we can prove that \eqref{868} holds for any sharp $\sigma$-admissible pair $(q,r)$ and any $\sigma$-admissible pair $(\tilde{q},\tilde{r})$ such that $(q,r)\neq P_{\sigma}$ and $(\tilde{q},\tilde{r})\neq P_{\sigma}$. Again, since each $\sigma$-admission pair $(q,r)$ satisfying $(q,r)\neq P_{\sigma}$ is an interpolation between $(\infty,r)$ and any sharp $\sigma$-admission pair $(q_0,r)$ with $(q_0,r)\neq P_{\sigma}$, it follows from \eqref{477}, \eqref{868} with $q=q_0=\frac{2r}{\sigma(r-2)}$, and Riesz's interpolation theorem that the inequality
\begin{equation*}
\bigg\|\int^{t}_{-\infty}{ V(t-s) F(s)}ds\bigg\|_{L^{q}_{t}L^{r}_{x}}\lesssim (C_2^{\frac{1}{2}}C_{4,r})^{1-\theta}C^{\theta}_{5,r,\tilde{r}}\mu_{1}^{(\frac{1}{2}-\frac{1}{\tilde{r}})(1-\theta)+\frac{1}{\sigma}(\frac{1}{q_0}+\frac{1}{\tilde{q}})\theta}\mu_{2}^{-(\frac{1}{\tilde{q}}+\frac{1}{q})}\| F\|_{L^{\tilde{q}'}_{t}L^{\tilde{r}'}_{x}}
\end{equation*}
holds for all $\sigma$-admissible pairs $(q,r)$ and $(\tilde{q},\tilde{r})$ such that $(q,r)\neq P_{\sigma}$ and $(\tilde{q},\tilde{r})\neq P_{\sigma}$, where $\theta$ is defined by \eqref{8611}. Thus \eqref{8123} is proved. Now we have completed the proof of Proposition \ref{a1}.
\end{proof}

We remark here that, by using the method in \cite{Tao}, estimates that are similar to \eqref{862}, \eqref{865} and \eqref{8123} can also be constructed for the endpoint $P_{\sigma}$ as $\sigma>1$. However, to avoid the redundancy we are not going to discuss it further. We now give some examples of the applications of Proposition \ref{a1}.

\begin{example}
In the strong stratified limit for the 3D inviscid Boussinesq equations studied in \cite{B4265} (Lemma 2.3), the corresponding quantities are
\begin{equation*}
V(t)=U_N(t),~~\mu_1=C,~~\mu_2=N,~~\sigma=\frac{1}{2},~~d=3,
\end{equation*}
and $\psi$ is a real-valued function in $\mathcal{S}(\mathbb{R}^3)$ satisfying $\mathrm{supp}\,\psi\subset\{2^{-2}\leq|\xi|\leq 2^2\}$ and $\psi=1$ on $\{2^{-1}\leq |\xi|\leq2\}$. It can be checked that
\begin{equation*}
\|\mathcal{F}^{-1}(\psi^2)\|_{L^{1}_x}+\|\mathcal{F}^{-1}(\psi)\|_{L^{1}_x}+\|\mathcal{F}^{-1}(\psi)\|_{L^{(\frac{1}{2}+\frac{1}{r})^{-1}}_x}\leq C,
\end{equation*}
thus from \eqref{862} and \eqref{8123} in Proposition \ref{a1} we obtain, for some constants $C$ that are independent of $f$, $F$ and $N$, that
\begin{equation*}
\|U_N(t)f\|_{L^q_t L^r_x}\leq CN^{-\frac{1}{q}}\|f\|_{L^2_x},
\end{equation*}
and
\begin{equation*}
\|\int^t_{-\infty}U_N(t-s)F(s)\|_{L^q_tL^r_x}\leq CN^{-\frac{1}{q}-\frac{1}{\tilde{q}}}\|F\|_{L^{\tilde{q}'}_tL^{\tilde{r}'}_x},
\end{equation*}
where the components $q,r,\tilde{q},\tilde{r}$ satisfy
\begin{equation*}
2\leq r,\tilde{r}\leq \infty,~~4\leq q,\tilde{q}\leq\infty,~~\frac{2}{q}\leq\frac{1}{2}\Big(1-\frac{2}{r}\Big),~~\mathrm{and}~~\frac{2}{\tilde{q}}\leq\frac{1}{2}\Big(1-\frac{2}{\tilde{r}}\Big).
\end{equation*}
\end{example}

\begin{example}
In the singular limit problems of the two-layer shallow water equations studied in \cite{Mu} (Proposition 3.2), we have
\begin{equation*}
V(t)=\Lambda_k(t),~~\mu_1=2^{2k},~~\mu_2=\frac{2^k}{\delta},~~\sigma=\frac{1}{2},~~d=2,~~\mathrm{and}~~\psi(\xi)=\tilde{\varphi}_k(\xi),
\end{equation*}
where $\tilde{\varphi}_k(\xi)=\tilde{\varphi}(2^{-k}\xi)$ with $\tilde{\varphi}$ being a real-valued function in $\mathcal{S}(\mathbb{R}^2)$ satisfying $\mathrm{supp}\,\tilde{\varphi}\subset\{2^{-1}\leq|\xi|\leq 3\}$ and $\tilde{\varphi}=1$ on $\{3/4\leq |\xi|\leq8/3\}$. Noting that
\begin{equation*}
\|\mathcal{F}^{-1}(\tilde{\varphi}_k^2)\|_{L^{1}_x}+\|\mathcal{F}^{-1}(\tilde{\varphi}_k)\|_{L^{1}_x}\leq C,~~\mathrm{and}~~\|\mathcal{F}^{-1}(\tilde{\varphi}_k)\|_{L^{(\frac{1}{2}+\frac{1}{r})^{-1}}_x}\leq C2^{2k(\frac{1}{2}-\frac{1}{r})},
\end{equation*}
thus, by \eqref{862} and \eqref{865} in Proposition \ref{a1} we derive, for some constants $C$ that are independent of $f$, $F$, $\delta$ and $k$, that
\begin{equation*}
\|\Lambda_k(t)f\|_{L^q_t L^r_x}\leq C2^{k(1-\frac{2}{r}-\frac{1}{q})}\delta^{\frac{1}{q}}\|f\|_{L^2_x},
\end{equation*}
and
\begin{equation*}
\|\int^t_{-\infty}\Lambda_k(t-s)F(s)\|_{L^q_tL^r_x}\leq C2^{k(1-\frac{2}{r}-\frac{1}{q}+\frac{3}{\tilde{q}})}\delta^{\frac{1}{q}+\frac{1}{\tilde{q}}}\|F\|_{L^{\tilde{q}'}_tL^{\tilde{r}'}_x},
\end{equation*}
where the components $q,r,\tilde{q},\tilde{r}$ satisfy
\begin{equation*}
2\leq r,\tilde{r}\leq \infty,~~4\leq q\leq\infty,~~\frac{2}{q}\leq\frac{1}{2}\Big(1-\frac{2}{r}\Big),~~\mathrm{and}~~\frac{2}{\tilde{q}}=\frac{1}{2}\Big(1-\frac{2}{\tilde{r}}\Big).
\end{equation*}
\end{example}

\section*{Compliance with ethical standards}
\textbf{Funding.} The author is supported by National Natural Science Foundation of China (Grant No. 12301279), Anhui Provincial Natural Science Foundation (Grant No. 2308085QA03), and Excellent University Research and Innovation Team in Anhui Province (Grant No. 2024AH010002).\\
\\
\textbf{Data availability statement.} This paper has no associated data.\\
\\
\textbf{Conflict of interest.} The author declares that there is no conflict of interest.

\end{document}